\documentclass[11pt]{amsart}
\usepackage[margin=2.5cm]{geometry}
\usepackage{amssymb}
\usepackage{bm}
\usepackage[centertags]{amsmath}
\usepackage{amsfonts}
\usepackage{amsthm}
\usepackage{mathrsfs}
\usepackage{mathtools,xfrac}

\usepackage[usenames,dvipsnames,svgnames,table]{xcolor}

\DeclareMathOperator\mtsupp{mt-supp}
\DeclareMathOperator\bdsupp{bd-supp}
\DeclareMathOperator\w{w}
\DeclareMathOperator\bdp{bd}
\DeclareMathOperator\mt{mt}
\DeclareMathOperator\age{age}
\DeclareMathOperator\ran{rng}
\DeclareMathOperator\rank{rank}
\DeclareMathOperator\supp{supp}
\DeclareMathOperator\rng{rng}
\DeclareMathOperator{\rngf}{rng_{\text{\tiny FDD}}}
\DeclareMathOperator{\Net}{Net}

\DeclareMathOperator{\dist}{dist}

\numberwithin{equation}{section}

\newtheorem{theorem}{Theorem}[section]
\newtheorem{lemma}[theorem]{Lemma}
\newtheorem{proposition}[theorem]{Proposition}
\newtheorem{corollary}[theorem]{Corollary}

\newtheorem{definition}[theorem]{Definition}

\newtheorem{notation}{Notation}

\newtheorem{remark}[theorem]{Remark}

\DeclarePairedDelimiter\abs{\lvert}{\rvert}
\newcommand{\e}{\delta}
\newcommand{\ee}{\epsilon}

\newcommand{\norm}[1][\cdot]{\lVert #1 \rVert}
\newcommand{\wth}[1][\eta]{\widetilde{ #1 }}

\newcommand{\N}{\mathbb{N}}
\newcommand{\R}{\mathbb{R}}
\newcommand{\F}{\mathcal{F}}
\newcommand{\E}{\mathcal{E}}
\newcommand{\T}{\mathcal T}

\newcommand{\esga}[1][\gamma]{e^{*}_{#1}}

\newcommand{\egs}[1][\gamma]{e^{*}_{#1}}
\newcommand{\dg}[1][\gamma]{d_{#1}}

\newcommand{\BmT}{\mathfrak{B}_{\mathrm{mT}}}
\newcommand{\Ga}{\Gamma}

\newcommand{\brd}{\bar{d}}
\newcommand{\brc}{\bar{c}}

\newcommand{\bP}{\bar{P}}
\newcommand{\bDelta}{\bar{\Delta}}
\newcommand{\bGamma}{\bar{\Gamma}}
\newcommand{\bGammaq}[1][q+1]{\bar{\Gamma}_{#1}}
\newcommand{\bDeltaq}[1][q+1]{\bar{\Delta}_{#1}}
\newcommand{\ageg}[1][\gamma]{\mathrm{age}(#1)}
\newcommand{\bdgs}[1][\gamma]{\bar{d}^{*}_{#1}}

\title{An unconditionally saturated Banach space with the scalar-plus-compact 
property}
 \author{Antonis Manoussakis} 
\address{School of Environmental Engineering, Technical University of Crete, 
 University Campus, 73100 Chania, Greece}
 \email{amanousakis@isc.tuc.gr}
 \author{Anna Pelczar-Barwacz} 
 \address{Institute of Mathematics, Faculty of Mathematics and Computer 
Science, 
Jagiellonian University, {\L}ojasiewicza 6, 30-348 Krak\'ow, Poland}
 \email{anna.pelczar@im.uj.edu.pl}
 \author{Micha{\l} \'Swi\c{e}tek}
 \address{Institute of Mathematics, Faculty of Mathematics and Computer 
Science, 
Jagiellonian University, {\L}ojasiewicza 6, 30-348 Krak\'ow, Poland}
 \email{michal.swietek@uj.edu.pl}
\thanks{The fisrt author was partially supported by 
Bilateral Educational Programme, Greece-Poland 2015}
\begin{document}
\maketitle
\begin{abstract}
 We construct a Bourgain-Delbaen $\mathscr{L}_\infty$-space $\mathfrak{X}_{Kus}$ with structure that is  strongly heterogeneous: any bounded operator on $\mathfrak{X}_{Kus}$ is a compact perturbation of a multiple of the identity, whereas the space $\mathfrak{X}_{Kus}$ is saturated with unconditional basic sequences. 
\end{abstract}

\section{Introduction}
J.~Bourgain and F.~Delbaen presented in \cite{BD} a brilliant method of constructing $\mathscr{L}_{\infty}$-spaces with peculiar structure. Their method relies on a careful choice of an increasing sequence of finite dimensional subspaces $(F_n)_n$ of $\ell_\infty(\Gamma)$, with countably infinite $\Gamma$ and each $F_n$ uniformly isomorphic to $\ell_{\infty}^{dim F_{n}}$. A suitable choice of $(F_n)_n$ guarantees that the space $\overline{\cup_nF_n}$ is an $\mathscr{L}_\infty$-space with no unconditional basis. The Bourgain-Delbaen example contains no isomorphic copy of $c_0$, answering an old problem in the theory of $\mathscr{L}_\infty$-spaces. Later R.~Haydon \cite{H} proved that this space is saturated with reflexive $\ell_p$ spaces and introduced the notation used nowadays. The Bourgain-Delbaen method was used to construct Banach spaces that solved several other long-standing conjectures on the structure of Banach spaces and showed that one may not hope for an ordinary classification of $\mathscr{L}_{\infty}$-spaces as it happens in the $C(K)$-spaces case, see \cite{AFHO}, \cite{AGM}, \cite{AH}, \cite{FOS}. We refer to \cite{B} and \cite{BD} for the properties of the classical Bourgain-Delbaen spaces.

In \cite{AGM} a general Bourgain-Delbaen-$\mathscr{L}_\infty$-space is defined and the authors show a remarkable fact that any separable $\mathscr{L}_{\infty}$-space is isomorphic to such a space. We recall from \cite{AGM} that a BD-$\mathscr{L}_\infty$-space is a space $\mathfrak{X}\subset\ell_\infty(\Gamma)$, with $\Gamma$ countable, associated to a sequence $(\Gamma_q,i_q)_{q\in\N}$, where $(\Gamma_q)_q$ is an increasing sequence of finite sets with $\Gamma=\cup_{q\in\N}\Gamma_q$ and $(i_q)_q$ are uniformly bounded compatible extension operators $i_q:\ell_\infty(\Gamma_q)\to\ell_\infty(\Gamma)$, i.e. $i_q(x)|_{\Gamma_q}=x$ and $i_q(x)=i_p(i_q(x)|_{\Gamma_p}))$ for any $q<p$ and $x\in\ell_\infty(\Gamma_q)$. The space $\mathfrak{X}=\mathfrak{X}_{(\Gamma_q,i_q)_q}$ is defined as $\mathfrak{X}=\overline{\langle d_\gamma: \gamma\in\Gamma\rangle}$, where $d_\gamma$ is given by $d_\gamma=i_q(e_\gamma)$, with $q$ chosen so that $\gamma\in\Gamma_q\setminus\Gamma_{q-1}$. An efficient method of defining particular examples of BD-$\mathscr{L}_\infty$-spaces as quotients of canonical BD-$\mathscr{L}_\infty$-spaces was given in \cite{AMo}. The authors proved that given a BD-$\mathscr{L}_\infty$-space $\mathfrak{X}\subset\ell_\infty(\Gamma)$ any so-called self-determined set $\Gamma^\prime\subset\Gamma$ produces a further $\mathscr{L}_\infty$-space $Y=\overline{\langle d_\gamma: \gamma\in\Gamma\setminus \Gamma^\prime\rangle}$ and a BD-$\mathscr{L}_\infty$-space $\mathfrak{X}/Y$, with the quotient map defined by the restriction of $\Gamma$ to $\Gamma^\prime$.

S.A.~Argyros and R.~Haydon in \cite{AH} used the Bourgain-Delbaen method in order to produce an $\mathscr{L}_{\infty}$-space $\mathfrak{X}_{AH}$ which is hereditary indecomposable (HI) i.e. contains no closed infinitely dimensional subspace which is a direct sum of further two closed infinitely dimensional subspaces (in particular the space $\mathfrak{X}_{AH}$ admits no unconditional basic sequence), and with dual isomorphic to $\ell_{1}$. Moreover, using in an essential way the local unconditional structure imposed by the $\ell_{\infty}^{dim F_{n}}$-spaces they  proved that the space $\mathfrak{X}_{AH}$ has the scalar-plus-compact property i.e. every bounded operator on the space is of the form $\lambda I+K$, with $K$ compact and $\lambda$ scalar. 

Although it readily follows that there does not exist a Banach space with an unconditional basis and the scalar-plus-compact property, the latter property does not exclude rich unconditional structure inside the space. This is witnessed in \cite{AFHO}, where it was shown that, among other spaces, any separable and uniformly convex Banach space embeds into an $\mathscr{L}_{\infty}$-space with the scalar plus compact property. Therefore, a naturally arising question is whether there exists a Banach space with the scalar-plus-compact property that is saturated with unconditional basic sequences. 

Recall here that the first example of a space with an unconditional basis and a small family of operators is due to W.T.~Gowers, who ''unconditionalized'' in \cite{G} the famous Gowers-Maurey space, \cite{GM1}, producing a space $X_{G}$ with unconditional basis that solved the hyperplane problem. Afterwards, W.T.~Gowers and B.~ Maurey, \cite{GM2}, proved that any bounded operator on the space $X_{G}$ is of the form $D+S$, with $D$ diagonal and $S$ strictly singular. Gowers asked if an analogous property holds for the operators defined on subspaces of $X_{G}$ and if such property characterises a class of so-called tight by support Banach spaces, as it is in the case of complex HI spaces according to \cite{F}. This question was answered negatively by the first two named authors \cite{MP}. 

An example of a space with rich unconditional structure and a small family of bounded operators of a different type was presented in \cite{AM}, where the authors built a Banach space saturated with unconditional sequences and satisfying the following property: any bounded operator on the space is a strictly singular perturbation of a multiple of identity (recall that an operator is strictly singular provided none of its restriction to an infinitely dimensional subspace is an isomorphism onto its range). The construction used the saturated norms technique in a mixed Tsirelson space setting. 

In this paper we continue the study of Banach spaces with a small family of operators by showing the existence of a Banach space with a strongly heterogeneous structure. More precisely we construct a BD-$\mathscr{L}_\infty$-space $\mathfrak{X}_{Kus}$ with a basis satisfying the following properties:
\begin{enumerate}
\item Any bounded operator $T:\mathfrak{X}_{Kus}\to \mathfrak{X}_{Kus}$ is of the form $T=\lambda Id_{\mathfrak{X}_{Kus}}+K$, with $K$ compact and $\lambda$ scalar.
\item The space $\mathfrak{X}_{Kus}$ is saturated with unconditional basic sequences.
\item The dual space to $\mathfrak{X}_{Kus}$ is isomorphic to $\ell_1$.
\end{enumerate}

The structure of the space of bounded operators $\mathcal{B}(\mathfrak{X}_{Kus})$ implies that the space $\mathfrak{X}_{Kus}$ is indecomposable, however, being unconditionally saturated, it admits no HI structure. The space $\mathfrak{X}_{Kus}$ is thus the first example of a Banach space with the scalar-plus-compact property failing to have any HI structure. Let us recall that M.~Tarbard in \cite{T} constructed an indecomposable BD-$\mathscr{L}_\infty$-space $\mathfrak{X}_\infty$, that is not HI, but the Calkin algebra $\mathcal{B}(\mathfrak{X}_\infty)/\mathcal{K}(\mathfrak{X}_\infty)$ is isomorphic to $\ell_{1}$.

In order to build $\mathfrak{X}_{Kus}$ we adapt the idea of a construction of a Banach space $X_{ius}$ of \cite{AM} to the scheme of the Argyros-Haydon construction of Bourgain-Delbaen spaces \cite{AH}. This framework allows to pass from strictly singular operators to compact ones, however, in order to profit from this key property of the Argyros-Haydon construction we need to strengthen some results of \cite{AM} in the following way: we prove that if a bounded operator on the space converges to zero on the basis, then it converges to zero on any element of a special class of basic sequences, called RIS, instead of a saturating family of RIS (Prop. \ref{d-to-RIS}). In order to avoid a technical inductive construction of the space $\mathfrak{X}_{Kus}$ we follow the scheme of \cite{AMo}, defining $\mathfrak{X}_{Kus}$ as a suitable quotient of some variation of the canonical BD-$\mathscr{L}_\infty$-space $\mathfrak{B}_{mT}$ defined in \cite{AH}.

The balance between unconditional saturation and the restricted form of bounded operators on the whole space in the case of $X_{ius}$ was guaranteed by the form of so-called special functionals - the major tool in the construction of saturated norms. Any special functional in the norming set of $X_{ius}$ is a weighted average of a sequence of functionals, where the odd parts are weighted averages of the basis. Roughly speaking, the choice of the next functional of the weighted average is determined by the previously chosen odd parts and supports of the even parts. The freedom on the side of even parts allows changing signs of parts of even functionals of the weighted average, which in turn provides saturation by unconditional sequences. On the other hand, the control over the supports of the even parts guarantees the typical property of such construction, i.e. in our case given two RIS $(x_n)$ and $(y_n)$ with pairwise disjoint ranges and $\epsilon>0$ one is able to built on $(y_n)$ an average $\sum_na_ny_n$ of norm 1, such that $\norm[\sum_na_nx_n]<\epsilon$. This last property is crucial for proving the form of a bounded operator on a space. 

The direct translation of the special functionals described above into the setting of BD-spaces is impossible, as any change of signs of a part of a norming functional changes its support. In order to overcome this obstacle we use in the definition of functionals on the space $\mathfrak{X}_{Kus}$ projections on finite intervals instead of projections on right intervals of the form $[p,\infty)$ (Section 2.1) and substitute the equality of supports of even parts of special functionals by tight relation between tree-analysis of even parts (definition of special nodes, Section 5). The latter notion in the setting of the Argyros-Haydon construction comes from \cite{GPZ} and proves to be a very efficient tool in our case.

The paper is organized as follows: in Section 2. we describe the construction of the general space we shall use, including different kinds of analyses of norming functionals. Section 3. is devoted to the properties the basis, including the notion of neighbour nodes, within the general framework. In Section 4. we give the definition of $\mathfrak{X}_{Kus}$. In Section 5. and 6. we study the rapidly increasing sequences (RIS) and the dependent sequences respectively. Section 7. contains the results on bounded operators on the space, whereas Section 8. - the proof of unconditional saturation.

We are  grateful to Spiros Argyros  and Pavlos  Motakis for suggesting using the approach to defining BD-$\mathscr{L}_\infty$-spaces of \cite{AMo} which greatly simplified presentation of the definition of the space $\mathfrak{X}_{Kus}$.

\section{The base BD-$\mathscr{L}_\infty$-space $\mathfrak{X}_{\bGamma}$}\label{gensc}

We present in this section a BD-$\mathscr{L}_\infty$-space $\mathfrak{X}_{\bGamma}$, which is a minor modification of the space $\BmT$ defined in \cite{AH}. We shall define later the space $\mathfrak{X}_{Kus}$ as determined by some set $\Gamma\subset\bGamma$ following the general scheme of \cite{AMo}.

\subsection{Definition}

Pick $(m_{k})_{k},(n_{k})_{k},(l_{k})_k \nearrow +\infty$ such that $m_{1}=4, n_{1}=4, l_1=2$ and 
\begin{equation}\label{mklk}
m_{k}m_{k-1}\leq m_{1}^{l_{k}}\quad (\frac{n_{k-1}}{m_{k-1}})^{l_{k}}\leq \frac{n_{k}}{m_{k-1}m_{k}}, \ \ k\in\N.
\end{equation}
For example take $(2^{2^k})_k,(2^{2^{k^2}})_k,(2^k)_k$.

Following \cite{AH} we shall define recursively finite sets of nodes $\bDelta_q$ and $\bGamma_q=\bDelta_1\cup\dots\cup\bDelta_{q}$, $q\in\N$. Along with each set $\bDelta_{q}$ we define functionals $(\brc^*_\gamma)_{\gamma\in\bDelta_{q}}\subset\ell_1(\bGamma_q)$ and further $(\brd^*_\gamma)_{\gamma\in\bDelta_{q}}\subset\ell_1(\bGamma_{q})$ as $\brd^*_\gamma=e^*_\gamma-\brc^*_\gamma$. Having defined all sets $\bDelta_q$, $q\in\N$, we let $\bGamma=\cup_q\bGamma_q$.

We proceed now to the inductive construction. We let $\bDeltaq[1]=\{ 1\}$, $c_1^{*}=0$ and thus $\brd^{*}_{1}=e^{*}_1$.

Assume we have defined sets $\bDelta_1,\dots, \bDelta_{q}$. By $(e^*_\gamma)_{\gamma\in\bGamma_q}$ we denote the standard unit vector basis of $\ell_1(\bGamma_q)$. We enumerate the set $\bDelta_q$ using $\{\# \bGamma_{q-1}+1, \dots,\# \bGamma_q\}$ as the index set and in the set $\bGamma_{q}$ we consider the corresponding enumeration. Thus we can regard the sets $\bDelta_q$ and $\bGamma_q$ as intervals of $\N$. 
We use the notation $(\gamma_n)_n$ to refer to this enumeration.

For any interval $I\subset \bGammaq[q]$ let $\bar{P}^{*}_I$ be the projection onto $\langle \bar{d}^{*}_{\gamma_n}: n\in I\rangle$. For simplicity for any $n\in\N$ by $\bar{P}^{*}_n$ we denote the projection $\bar{P}^{*}_{(0,n]}$.

For each $q\in\N$ let $\Net_{1,q}$ be a finite symmetric $\sfrac{1}{4n_{q}^{2}}$-net of $[-1,1]$ containing $\pm 1$. We set 
$$
B_{p,q}=\{\lambda\esga[\eta]: \lambda\in \Net_{1,q}, \eta\in\bGamma_{q}\setminus\bGamma_{p}\},
$$
where for $p=0$ we let $\bGamma_{0}=\emptyset$. For simplicity we write $B_{q}=B_{0,q}$, $q\in\N$.

The set $\bDeltaq$ is defined to be the set of nodes
\begin{align*}
\bDeltaq&=\bigcup_{j=1}^{q}
\{
(q+1, 0, m_{j},I, \epsilon, b^{*}): 
I \text{ interval }\subset \bGammaq[q],\epsilon\in\{-1,1\},
 b^{*}\in B_{q}\,\,\textrm{and}\,\,
\bar{P}^{*}_Ib^{*}\neq 0
\}
\\
&
\cup\bigcup_{1\le p<q} \bigcup_{j=1}^{p} 
\{
(q+1, \xi, m_{j}, I, \epsilon, b^{*}): \xi\in \bDeltaq[p], \w(\xi)=m_{j}^{-1}, \age(\xi)<n_{j}, \\
&
\hspace{3cm}\epsilon\in\{-1,1\}, b^{*}\in B_{p,q}, 
I \text{ 
interval }\subset \bGammaq[q]\setminus\bGammaq[p], 
\bar{P}^{*}_Ib^{*}\neq 0
\}
.
\end{align*}
For any $\gamma\in\bDelta_q$ we define $\brc^*_\gamma$ as follows.
\begin{equation}\label{cgamma1}
\brc_{\gamma}^{*}=
\begin{cases}
\frac{1}{m_{j}}\epsilon \bar{P}^{*}_I b^{*} &\textrm{for }\,\,\gamma=(q+1, 0, m_j,I,\epsilon,b^{*})
\\
e^{*}_{\xi}+\frac{1}{m_{j}}\epsilon \bar{P}^{*}_Ib^{*} &\textrm{for }\,\,\gamma=(q+1, \xi, m_{j},I,\epsilon, b^{*})
\end{cases}
\end{equation}
We let also $\brd^*_\gamma=e^*_\gamma-\brc^*_\gamma$.

\begin{notation}
For any $\gamma=(q+1, 0, m_{j},I, \epsilon, b^{*})$ we define $\age(\gamma)=1$ and for $\gamma=(q+1, \xi, m_{j},I, \epsilon, b^{*})$ we define $\age(\gamma)=\age(\xi)+1$. For any $\gamma=(q+1, 0, m_{j},I, \epsilon, b^{*})$ or $\gamma=(q+1, \xi, m_{j},I, \epsilon, b^{*})$ we define $\rank(\gamma)=q+1$ and weight $\w(\gamma)=m_j^{-1}$. 
\end{notation}
\begin{remark}
The main difference with the construction from \cite{AH} is that in the $q$-th step instead of taking $b^{*}$ from the net of the unit ball of the suitable $\ell_1(\bGamma_q\setminus \bGamma_p)$, we take $b^{*}$ only of the form $\epsilon\lambda e^{*}_\eta$, where $\epsilon=\pm 1$, $\lambda$ belongs to the suitable net of $[-1,1]$, and $\eta\in\bGamma_q\setminus\bGamma_p$. Moreover we allow projections on all intervals $I\subset \bGamma_q\setminus \bGamma_p$, while in \cite{AH} the allowable intervals are of the form $I=\bGammaq[q]\setminus\bGammaq[p]$. 
\end{remark}
Adapting the reasoning of \cite{AH} we obtain the following two lemmas.
\begin{lemma}\label{lemma1} 
$\langle \brd^{*}_{\gamma_{i}}:i\leq n\rangle=\langle e^{*}_{\gamma_{i}}:i\leq n\rangle$ for every $n\in\N$.
\end{lemma}

\begin{lemma}\label{lemma2} $\norm[\bP^{*}_{m}]\leq \frac{m_{1}}{m_{1}-2}=2 $ for every $m\in\N$.
\end{lemma}

The above lemma yields that $(\brd^{*}_{\gamma_{n}})_{n\in\N}$ is a triangular basis of $\ell_{1}(\Gamma)$ (in the sense of \cite{AH}, Def. 3.1). Let $(\brd_{\gamma_n})_{n\in\N}$ be its biorthogonal sequence. Regarding each projection $\bP_n^{*}$ as an operator $\ell_1(\Gamma)\to\ell_1^n$ we consider the dual operator $\bar{i}_n:\ell_\infty^n\to\ell_\infty(\Gamma)$, which is an isomorphic embedding satisfying $\norm[\bar{i}_n]\leq 2$. We are ready to define the following.
\begin{definition}
Let $\mathfrak{X}_{\bGamma}=\langle \brd_{\gamma_{n}}:n\in\N\rangle\subset\ell_\infty(\bGamma)$.
\end{definition}

Repeating the results of \cite{AH} in our setting we obtain the following.
\begin{theorem}
 The space $\mathfrak{X}_{\bGamma}$ is a BD-$\mathscr{L}_{\infty}$-space defined by the sequence $(\bGamma_q,\bar{i}_q)_q$. 
\end{theorem}

\begin{notation} 
For any interval $I\subset \N$ we denote by $\bP_I$ the canonical projection $\bP_I: \mathfrak{X}_{\bGamma}\to\langle \brd_{\gamma_i}: i\in I\rangle$. In case $I=\{1,\dots, n\}$, $n\in\N$, we write simply $\bP_n$.

Given any $q\in\N$ we let $\bar{M}_q=\bar{i}_{\max \bDelta_q}[\ell_\infty(\bDelta_q)]$. In the rest of the paper we shall consider supports and ranges of vectors, thus also block sequences, with respect both to the basis $(\brd_{\gamma_n})_{n\in\N}$ of $\mathfrak{X}_{\bGamma}$ and to the FDD $(\bar{M}_q)_{q\in\N}$ of $\mathfrak{X}_{\bGamma}$. In the first case we shall use for any $x\in \mathfrak{X}_{\bGamma}$ the notation $\supp x$, $\rng x$, whereas in the second we write $\supp_{FDD} x$ and $\rng_{FDD} x$. 
\end{notation}

\begin{definition}
We say that a block sequence $(x_n)_{n}\subset \mathfrak{X}_{\bGamma}$ is skipped provided $\max\rng_{FDD} x_n+1<\min\rng_{FDD}x_{n+1}$ for each $n$. 
\end{definition}

\subsection {The analysis of nodes}

We introduce different types of analysis of a node following \cite{AH} and \cite{GPZ}, adjusting their scheme to our situation. 

\textbf{The evaluation analysis of $e_{\gamma}^{*}$}.

First we notice that every $\gamma\in\bGamma$ admits a unique analysis as follows (Prop. 4.6 \cite{AH}). Let $\w(\gamma)=m_{j}^{-1}$. Then using backwards induction we determine a sequence of sets $(I_i,\epsilon_i, b_{\eta_i}^{*},\xi_i)_{i=1}^{a}$ so that $\xi_a=\gamma$, $\xi_1=(q_1+1,0,m_j, I_1,\epsilon_1,b_{\eta_1}^{*})$ and $\xi_i=(q_i+1,\xi_{i-1},m_j,I_i,\epsilon_i,b^{*}_{\eta_i})$ for every $1<i\leq a$, where $b_{\eta_{i}}^{*}=\lambda_{i}\esga[\eta_{i}]$ for some $\lambda_{i}\in\Net_{1,q_{i}}$.

Repeating the reasoning of \cite{AH}, as $e^{*}_{\xi}=\brd^{*}_{\xi}+c^{*}_{\xi}$ for each $\xi\in\Gamma$, with the above notation we have
$$
e^{*}_{\gamma}= \sum_{i=1}^{a}\brd^{*}_{\xi_{i}}+m_{j}^{-1}\sum_{i=1}^{a}\ee_{i} \bP^{*}_{I_{i}}b^{*}_{\eta_{i}} =\sum_{i=1}^{a}\brd^{*}_{\xi_{i}}+m_{j}^{-1}\sum_{i=1}^{a}\ee_{i} \lambda_{i}\bP^{*}_{I_{i}}e^{*}_{\eta_{i}}
$$

\begin{definition}
Let $\gamma\in\bGamma$. Then the sequence $(I_i,\epsilon_i,\lambda_{i}e_{\eta_i}^{*}, \xi_i)_{i=1}^{a}$ satisfying all the above properties will be called the evaluation analysis of $\gamma$.

We define the bd-part and mt-part of $e^{*}_\gamma$ as 
$$
\bdp(e^{*}_\gamma)=\sum_{i=1}^{a} \brd_{\xi_i}^{*}, \ \ \ \mt(e^{*}_\gamma)=m_{j}^{-1}\sum_{i=1}^{a}\ee_{i}\lambda_{i}\bP^{*}_{I_{ i} }e^{*}_{\eta_{i}}.
$$
\end{definition}

\begin{remark}\label{M}
For any $\xi\in\Gamma_q$ we have $\bP^{*}_{\Delta_{\rank(\xi)}}e^{*}_\xi=\brd^{*}_\xi$.
\end{remark}

\textbf{The $I$(interval)-analysis of a functional $e_{\gamma}^{*}$}.

Let $I\subset\N$ and $\gamma\in\Gamma$ with $\bP^{*}_Ie^{*}_\gamma\neq 0$. Let $\w(\gamma)=m_{j}^{-1}$, $a\leq n_{j}$ and $(I_i,\epsilon_i,\lambda_{i}e_{\eta_i}^{*}, \xi_i)_{i=1}^{a}$ the evaluation analysis of $\gamma$. We define the $I$-analysis of $e_{\gamma}^{*}$ as follows:
\begin{enumerate}

\item[(a)] If for at least one $i$ we have $\bP^{*}_{I_i\cap I}e^{*}_{\eta_i}\neq 0$, then the $I$-analysis of $e_{\gamma}^{*}$ is of the following form
$$
(I_i\cap I,\epsilon_i,\lambda_{i}e^{*}_{\eta_i},\xi_i)_{i\in A_I}, 
$$
where $A_I=\{i:\ \bP^{*}_{I_i\cap I}e^{*}_{\eta_i}\neq 0\}$. In this case we say that $e_{\gamma}^{*}$ is $I$-decomposable.
\item[(b)] If $\bP^{*}_{I_i\cap I}e^{*}_{\eta_i}=0$ for all $i=1,\dots,a$, then we assign no $I$-analysis to $e_{\gamma}^{*}$ and we say that $e_{\gamma}^{*}$ 
is $I$-indecomposable. 
\end{enumerate}

\begin{remark}\label{eve-ana}
 Notice that in the second case above, as $I$ is interval and $\bP^{*}_Ie^{*}_\gamma\neq 0$, $\bP^{*}_I e^{*}_\gamma= d^{*}_{\xi_{i_0}}$ for some $i_0\in \{1,\dots,a\}$. In other words, $e^{*}_\gamma$ is $I$-indecomposable iff $\bP^{*}_Ie^{*}_\gamma=\brd^{*}_\xi$ for some element $\brd^{*}_\xi$ of the bd-part of $e^{*}_\gamma$. 
\end{remark}
Now we introduce the tree-analysis of $e_{\gamma}^{*}$ analogous to the tree-analysis of a functional in a mixed Tsirelson space (see \cite{AT} Chapter II.1).

We start with some notation. We denote by $(\T,\preceq)$ a finite tree, whose elements are finite sequences of natural numbers ordered by the initial segment partial order. Given $t\in\T$ denote by $S_t$ the set of immediate successors of $t$.

Let $(I_t)_{t\in\T}$ be a tree of intervals of $\N$ such that $t\preceq s$ iff $I_t\supset I_s$ and $t,s$ are incomparable iff $I_t\cap I_s=\emptyset$. For such a family $(I_t)_{t\in\T}$ and $t,s$ incomparable we write $t<s$ iff $I_t<I_s$ (i.e. $\max I_t<\min I_s$).

\textbf{The tree-analysis of a functional $e_{\gamma}^{*}$}.

Let $\gamma\in\bGamma$. The tree-analysis of $e_{\gamma}^{*}$ is a family of the form $(I_t,\epsilon_t,\eta_t )_{t\in\T}$ defined inductively in the following way:
\begin{enumerate}
\item $\T$ is a finite tree with a unique root denoted by $\emptyset$.
\item Set $\eta_{\emptyset}=\gamma$, $I_{\emptyset}=(1,\max \Delta_{\rank \gamma}]$, $\epsilon_\emptyset=1$ and let $(I_i,\epsilon_i,\lambda_{i}e_{\eta_i}^{*}, \xi_i)_{i=1}^{a}$ be the evaluation analysis of $e^*_{\eta_{\emptyset}}$. Set $S_{\emptyset}=\{(1),(2),\ldots,(a)\}$ and for every $s=(i)\in S_{\emptyset}$, $(I_s,\epsilon_s,\eta_s)=(I_i,\epsilon_i,\eta_i)$.
\item Assume that for $t\in\T$ the tuple $(I_t,\epsilon_t,\eta_t)$ is defined. Let $(I_i,\epsilon_i,\lambda_{i}e^{*}_{\eta_i}, \xi_i)_i$ be the evaluation analysis of $e_{\eta_t}^{*}$. Consider two cases:
\begin{enumerate}
\item If $e_{\eta_t}^{*}$ is $I_t$-decomposable, let $(I_i,\epsilon_i,\lambda_{i}e^{*}_{\eta_i}, \xi_i)_{i\in A_{I_t}}$ be the $I_t$-analysis of $e_{\eta_t}^{*}$. Set $S_t=\{(t^\smallfrown i): i\in A_{I_t}\}$. For every $s=(t^\smallfrown i)\in S_t$, let $(I_s,\epsilon_s,\eta_s)=(I_i,\epsilon_i,\eta_i)$.
 \item If $e_{\eta_t}^{*}$ is $I_t$-indecomposable, then $t$ is a terminal node of the tree-analysis. 
\end{enumerate}
\end{enumerate}

\begin{definition}
Given any $\gamma\in\Gamma$, in notation of Remark \ref{eve-ana} let 
$$
\mtsupp e^{*}_{\gamma}=\{\xi_t: \ t\in \T, t \text{ terminal}\}= \{\xi_{t}: t\in \T,\,\,\,\, \bP^{*}_{I_t}e^{*}_{\eta_t}=\brd^{*}_{\xi_t}\}
$$ 
and $\bdsupp e^{*}_\gamma=\supp e^{*}_\gamma\setminus \mtsupp e^{*}_\gamma$.
\end{definition}

\section{Properties of the basis $(\brd_{\gamma_n})$}

We present here estimates on the averages of the basis $(\bar{d}_{\gamma_n})_{n\in\N}$. 

\subsection{Neighbours nodes}
The result of this section is crucial for the estimates in the sequel. 
\begin{definition}
We shall call two nodes $\xi_{1},\xi_{2}$ neighbours if there exists $\gamma\in \Gamma$ with $\bdp(e^{*}_\gamma)=\sum_{j=1}^{a}\brd_{\zeta_{j}}^{*}$ such that $\xi_{i}=\zeta_{j_{i}}$ for some $j_{1}<j_{2}$.
\end{definition}
Note that from the definition it follows that for any neighbours $\xi_{1},\xi_{2}$ we have $\w(\xi_{1})=\w(\xi_{2})$.

 \begin{lemma}\label{neighbours}
Let $(\brd_{\gamma_{n}})_{n\in N}$ be a subsequence of the basis. Then there exists infinite $M\subset N$ such that no two nodes $\gamma_n, \gamma_m$, $n,m\in M$, are neighbours.
\end{lemma}
The proof is based on the fact that the age is uniquely determined for each node.
\begin{proof}
If there are infinitely many nodes with different weights we are done. So assume that for all but finite nodes we have $\w(\gamma_{n})=m^{-1}_{k}$ for some fixed $k$.

Applying Ramsey theorem we obtain an infinite set such that either no two nodes from this set are neighbours or any two are neighbours.

In the first case we are done. Otherwise passing to a further subsequence we may assume that $\rank(\gamma_{n})<\rank(\gamma_{n+1})$ for every $n$.

Since we have that $\gamma_{j}, \gamma_{j+1}$ are neighbours it follows by a simple induction that
$$
\age(\gamma_{j+1})\geq \age(\gamma_{j})+1\geq j+1.
$$
Take $j=n_{k}+1$ and pick $e^{*}_\gamma$ of the form
$$
\egs=\sum_{r=1}^{a}\brd^{*}_{\xi_{r}}+m^{-1}_{k}\sum_{r=1}^{a}\ee_{r}\lambda_{r}\esga[\eta_{r}]\bP_{I_{r}}
$$
with $\bdgs[\gamma_{n_{k}+1}]=\bdgs[\xi_{r}]$ for some $r$. Then $\age(\xi_{r})\leq n_{k}$ which yields a contradiction and ends the proof.
\end{proof}

\subsection{Estimates on some averages of the basis}

In \cite{AH} it is proved that the sequence $(\sum_{\xi\in\Delta_{n}}\brd_{\xi})_{n\in\N}$ generates an $\ell_{1}$-spreading model in the space $\mathfrak{X}_{AH}$. We show that the norm of the vector $y=n_j^{-1}\sum_{i\in F}\brd_{\xi_i}$, where $\xi_i$'s are pairwise non-neighbours, is determined by the mt-part of the nodes.

In the sequel we shall use basic properties of mixed Tsirelson spaces. Recall that the mixed Tsirelson space $T[(\mathcal{A}_{n_k},m_{k}^{-1})_{k\in\N}]$ is the completion of $c_{00}$ with the norm defined by a norming set $D$, which is the smallest set in $c_{00}$ that contains the unit vectors $\{\pm e_n\}$ and satisfies for any $k\in\N$ the following condition: for any block sequence $f_1<\dots<f_d$, $d\leq n_k$, of elements of $D$ the weighted average $m_k^{-1}(f_1+\dots+f_d)$ also belongs to $D$. For further details see \cite{AT}.

\begin{lemma}\label{basisest} Let $x=n_j^{-1}\sum_{i\in G}\brd_{\xi_i}$, be such that no two $\xi_i$'s are neighbours and $\# G\leq n_{j}$. Then for any $\gamma\in\Gamma$ with $\w(\egs)=m_{k}^{-1}$ we have the following
$$
\abs{\egs(x)}\leq\begin{cases}
\frac{1}{n_{j}}+\frac{2}{m_{k}}\quad &\textrm{if}\,\,\,k\geq j
\\
\frac{7}{m_{k}m_{j}}\quad &\textrm{if}\,\,\,k<j,
\end{cases}
$$
In particular
$$
\norm[n_j^{-1}\sum_{i=1}^{n_j}\brd_{\xi_i}]\leq 7m_{j}^{-1}.
$$
\end{lemma}
\begin{proof}
We shall construct functionals $\phi_{\gamma}$ in the norming set of the mixed Tsirelson space $X_{aux}=T[(\mathcal{A}_{n_k},m_{k}^{-1})_{k\in\N}]$ such that 
$$
\abs{e_{\gamma}^{*}(x)}\leq \phi_{\gamma}(y)+\frac{2}{m_{j}m_{j-1}}
$$
where $y=2\sum_{k\in G}e_{k}/n_{j}\in c_{00}(\N)$.

Take $\gamma\in\Gamma$ and consider its evaluation analysis  $e_{\gamma}^{*}=\sum_{r=1}^{a}\brd_{\beta_{r}}^{*}+m_{k}^{-1}\sum_{r=1}^{a}
\epsilon_{r}\lambda_{r}\esga[\eta_{r}]\bP_{I_{r}}$. Let $g_\gamma=\bdp(e^{*}_\gamma)$ and $f_\gamma=\mt(e^{*}_\gamma)$.

We shall consider two cases.

\noindent\textbf{Case 1}. $\w(\gamma)\leq m_{j}^{-1}$.

Since the nodes $(\xi_{i})_{i}$ are pairwise non-neighbours and $(\beta_{i})_{i} $ are pairwise neighbours it follows that
\begin{equation}\label{1eq} 
 \abs{g_{\gamma}(x)}\leq n_{j}^{-1}.
\end{equation}
Also for every $r\leq a$ using that $|\esga[\zeta](\brd_{\beta})|\leq 2$ for all $\zeta,\beta$, we get
\begin{equation}\label{2eq}
\abs{\esga[\eta_{r}]\bP_{I_{r}}(x)} \leq 2\frac{\# \{i: \rng(d^*_{\xi_i})\subset I_{r}\} }{n_{j}}\,.
\end{equation}
It follows from \eqref{1eq},\eqref{2eq}, using that $|\lambda_{r}|\leq 1$ for every $r$, that
\begin{equation}\label{3eq}
\abs{e_{\gamma}^{*}(x)}
\leq
\frac{1}{n_{j}}+2m_{k}^{-1}\sum_{r=1}^{a}\frac{\# \{i\mid \rng(d^*_{\xi_i})\subset I_{r}\} }{n_{j}}\leq\frac{1}{n_{j}}+\frac{2}{m_{k}}.
\end{equation}
Taking $\phi_{\gamma}=m_{k}^{-1}\sum_{n\in F} e_{n}^{*}$ where $F=\cup_{r\leq a}\{n\mid \gamma_n=\xi_i, \rng(d^*_{\xi_i})\subset I_{r}\, \textrm{for some}\,i\in G\}$ it follows that $\# F\leq n_{j}\leq n_{k}$ and $\phi_{\gamma}$ belongs to the norming set of the mixed Tsirelson space $X_{aux}$.

From \eqref{3eq} we get 
\begin{equation}\abs{e_{\gamma}^{*}(x)}\leq
\frac{1}{n_{j}}+
2m_{k}^{-1}\sum_{n\in F}\frac{e_{n}^{*}(e_{n})}{n_j}=
\frac{1}{n_{j}}+\phi_{\gamma}(y).
\label{4eq}
\end{equation}
\textbf{Case 2.} $\w(\gamma)=m_{k}^{-1}> m_{j}^{-1}$.

Let $(I_t,\varepsilon_t,\eta_t)_{t\in\T}$ be the tree-analysis of $e^{*}_{\gamma}$ and $\T'$ be the subtree of $\T$ consisting of all nodes $t$ of height at most $l_j$. 
We will describe how to define certain functionals $(\phi_t)_{t\in\T'}$ in the norming set of $T[(\mathcal A_{n_k},m_k^{-1})_{k\in\N}]$ that we will use to obtain the desired estimate.

As in the previous case we get
\begin{equation}\label{5eq}
 \abs{g_{\gamma}(x)}\leq n_{j}^{-1}.
\end{equation}
Using that $\egs=g_{\gamma}+f_{\gamma}$ and $|\lambda_r|\leq 1$ for every $r$, we get
\begin{equation}\abs{e_{\gamma}^{*}(x)}\leq n_{j}^{-1}+\abs{f_{\gamma}(x)}
\leq 
n_{j}^{-1}+
m_{k}^{-1}\sum_{r=1}^{a}|\esga[\eta_{r}]\bP_{I_{r}}(x)|.
\label{6eq}
\end{equation}
We shall split now the successors $\esga[\eta_{r}]$ of $\egs$ into those with weight smaller or equal to $m_{j}^{-1}$ and those with weight bigger that $m_{j}^{-1}$. For a node $\gamma$ we set
$$ 
S_{\gamma,1}=\{r\in S_{\gamma}: w(\eta_{r})\leq 
m_{j}^{-1}\}\quad\textrm{and}\quad S_{\gamma,2}=S_{\gamma}\setminus S_{\gamma,1}.
$$
From \eqref{6eq} we get
\begin{align*}\abs{e_{\gamma}^{*}(x)}&\leq
n_{j}^{-1}+
m_{k}^{-1}\left( \sum_{r\in S_{\gamma,1}} \abs{\esga[\eta_{r}]\bP_{I_{r}}(x)}
+\sum_{r\in S_{\gamma,2} } 
\abs{\esga[\eta_{r}]\bP_{I_{r}}(x)}\right)
\end{align*}
Using \eqref{4eq} for the $r\in S_{\gamma,1}$,
 \eqref{6eq} for the $r\in S_{\gamma,2}$ 
 and that $\# S_{\gamma,1}+\# S_{\gamma,2}\leq n_{k}$, $k<j$, we get \begin{align}
\abs{e_{\gamma}^{*}(x)}
&
\leq
n_{j}^{-1}+
\frac{n_{k}}{m_{k}n_{j}}+
\frac{1}{m_{k}}
\left(
\sum_{r\in S_{\gamma,1} }\phi_{r}(y)
+\sum_{r\in S_{\gamma,2}} 
w(e_{\eta_{r}})\sum_{s\in S_{r} }
\abs{e_{\eta_{s}}^{*}\bP_{I_{s}}(x)} \right)
 \notag
 \\
 &\leq\frac{1}{ n_{j} }(1+\frac{n_{j-1} }{ m_{j-1} })+
 \frac{1}{m_{k} }
 \left(
 \sum_{r\in S_{\gamma,1} }\phi_{r}(y)
 +\sum_{r\in S_{\gamma,2}} 
w(e_{\eta_{r}})\sum_{s\in S_{r} }
\abs{e_{\eta_{s}}^{*}\bP_{I_{s}}(x)}
\right).
\label{7eq}
\end{align}
Note that the functional $m_{k}^{-1} \left( \sum_{r\in S_{\gamma,1} }\phi_{r}\right)$ belongs to the norming set of the mixed Tsirelson space $X_{aux}$ and has room for $\# S_{\gamma,2}$ more functionals.

We shall replay the above splitting for every $e_{\eta_{s}}^{*}\bP_{I_{s}} $. To avoid complicated notation we shall set $n_{s}=\# S_{s}$ and $m^{-1}_{s}=w(e_{\eta_{s}}^{*})$. From \eqref{7eq} using $e_{\eta_{s}}^{*}\bP_{I_{s}}$ in the place of $\esga$ we get
 \begin{align}\abs{e_{\eta_{s}}^{*}\bP_{I_{s}}(x) }\leq 
 \frac{1}{n_{j}}(1+\frac{n_{j-1}}{ m_{j-1} } )+
 m_{s}^{-1}
 \left(
 \sum_{t\in S_{s,1} }\phi_{t}(y)
 +\sum_{t\in S_{s,2}} 
 m_{t}^{-1}\sum_{u\in S_{t}}
 |e_{\eta_{u}}^{*}\bP_{I_{u}}(x)| 
 \right).
 \label{8eq}
 \end{align}
It follows that
\begin{align}\label{9eq}
\sum_{r\in S_{\gamma,2}} 
w(e_{\eta_{r}})\sum_{s\in S_{r} }&
\abs{e_{\eta_{s}}^{*}\bP_{I_{s}}(x)}
\leq
\sum_{r\in S_{\gamma,2}}m_{r}^{-1}
\sum_{s\in S_{r}}
 \frac{1}{n_{j}}(1+\frac{n_{j-1}}{ m_{j-1} } )
\\
&+
\sum_{r\in S_{\gamma,2}}m_{r}^{-1}
\sum_{s\in S_{r}}
m_{s}^{-1}
 \left(
 \sum_{t\in S_{s,1} }\phi_{t}(y)
 +\sum_{t\in S_{s,2}} 
 m_{t}^{-1}\sum_{u\in S_{t}}
 \abs{e_{\eta_{u}}^{*}\bP_{I_{u}}(x) }
 \right)
\notag
 \\
&\leq 
n_{k}
\frac{n_{r}}{m_{r}}\frac{1}{n_{j}}(1+\frac{n_{j-1}}{ m_{j-1} } )
\qquad\textrm{since $\# S_{\gamma,2}\leq n_{k}$ and $\# S_{r}\leq n_{r}$}
\notag
\\
&\quad +
\sum_{r\in S_{\gamma,2}}m_{r}^{-1}
\sum_{s\in S_{r}}
m_{s}^{-1}
 \left(
 \sum_{t\in S_{s,1} }\phi_{t}(y)
 +\sum_{t\in S_{s,2}} 
 m_{t}^{-1}\sum_{u\in S_{t}}
\abs{e_{\eta_{u}}^{*}\bP_{I_{u}}(x)}
 \right).
\notag
\end{align}
 By \eqref{7eq} and \eqref{9eq}, using that $\frac{n_{r}}{m_{r}}, 
\frac{n_{k}}{m_{k}}\leq\frac{n_{j-1}}{m_{j-1}}$ we get
 \begin{align}\label{11eq}
 &\abs{e_{\gamma}^{*}(x)}
 \leq
 \frac{1}{ n_{j} }(
 1+\frac{n_{j-1}}{ m_{j-1} } +
 (
 \frac{n_{j-1} }{ m_{j-1}}
 )^{2}
 + 
 (\frac{n_{j-1} }{ m_{j-1}}
 )^{3}
 )
 \\
 &+
 \frac{1}{m_{k} }
 \left(
 \sum_{r\in S_{\gamma,1} }\phi_{r}(y)
 +
 \sum_{r\in S_{\gamma,2}}m_{r}^{-1}
 \sum_{s\in S_{r}}
 m_{s}^{-1}
 \left(
 \sum_{t\in S_{s,1} }\phi_{t}(y)
 +\sum_{t\in S_{s,2}} 
 m_{t}^{-1}\sum_{u\in S_{t}}
 \abs{e_{\eta_{u}}^{*}\bP_{I_{u}}(x)} 
 \right)
 \right).
 \label{12eq}
 \end{align}
 Note that the functional
 $$
 \phi_{\gamma}=\frac{1}{m_{k} }
 \left(
 \sum_{r\in S_{\gamma,1} }\phi_{r}(y)
 +
 \sum_{r\in S_{\gamma,2}}m_{r}^{-1}
 \sum_{s\in S_{r}}
 m_{s}^{-1}
 \sum_{t\in S_{s,1} }\phi_{t}(y)
 \right)
 $$
 belongs to the norming set of the mixed Tsirelson space $X_{aux}$ and the functional $m_{s}^{-1} \sum_{t\in S_{s,1}}\phi_{t} $ has room for $\# S_{s,2}$ more functionals.

We continue this splitting at most $l_{j}$ times, see \eqref{mklk} for the choice of $l_{j}$, or till $S_{s,2}=\emptyset$ i.e. we do not have nodes with weight$>m_{j}^{-1}$.

If we stop before the $l_{j}$-th step we get that $\abs{\egs(x)}$ is dominated by $\phi_{\gamma}(y)$ plus the errors in \eqref{11eq}, where the sum end to the $l_{j}$-th power of $n_{j-1}/m_{j-1}$. Since $\phi_{\gamma}$ belongs to the norming set of the mixed Tsirelson space $X_{aux}$ it follows from \cite{AT}, Lemma II.9, that 
$$
\phi_{\gamma}(y)\leq 4m_{k}^{-1}m_{j}^{-1}.
$$
If we continue the splitting $l_{j}$-times, then there exists some node with $\w(\gamma_{t})>m_{j}^{-1}$. For every such node we have
$$
\left(\prod_{s\prec t}w(e_{\gamma_{s}})\right)
\abs{e_{\gamma_{t}}^{*}(x)}\leq 
(\frac{1}{m_{1}})^{l_{k}}
\abs{e_{\gamma_{t}}^{*}(x)}\leq
2m_{k}^{-1}m_{j}^{-1}\frac{\# \{i: \rng(d^*_{\xi_i})\subset
I_{t}\cap G\}}{n_{j}}
$$
since $m_{1}^{-l_{j}}\leq (m_{j}m_{j-1})^{-1}$, see \eqref{mklk}.

Summing the estimation of all those nodes we get upper estimate equal to $\sfrac{2 \# G}{m_{k}m_{j}n_{j}}\leq\sfrac{2}{m_{k}m_{j}}$.

The remaining nodes provide us with a functional in the norming set of the mixed Tsirelson space $X_{aux}$. By \cite{AT} its action on $y$ is bounded by $4m_{k}^{-1}m_{j}^{-1}$.

It remains to handle the errors \eqref{11eq}. In each case we have
$$
\frac{1}{n_{j}}
(1+\frac{n_{j-1}}{m_{j-1}} 
+(\frac{n_{j-1}}{m_{j-1}})^{2}+\dots+(\frac{n_{j-1}}{m_{j-1}})^{l_{j}})
\leq
\frac{1}{n_{j}}\frac{ 
(n_{j-1}/m_{j-1})^{l_{j}+1}-1}{(n_j/m_{j-1})-1}\leq\frac{2}{m_{j}m_{j-1}}.
$$
Summing all the above estimates we get an upper estimate $7m_{k}^{-1}m_{j}^{-1}.$
\end{proof}

\begin{remark}
Analogous estimates for the averages of the basis hold by the same argument in other spaces built in the Argyros-Haydon scheme of Bourgain-Delbaen construction.
\end{remark}
\begin{corollary}\label{diffwb}
Let $x=m_{j}n_{j}^{-1}\sum_{i=1}^{n_{j}}\brd_{\xi_{i}}$ such that no two $\xi_{i}$'s are neighbours. Let $i<j$, $(e^{*}_{\eta_{p}})_{p=1}^{n_{i}}$ be nodes such that $\w(e^{*}_{\eta_{p}})=m_{l_p}\ne m_{j}$ and $m_{l_{p}}<m_{l_{p+1}}$ for all $p\leq n_{i}$. Then
\begin{equation}\label{ediffwb}
\sum_{p=1}^{n_{i}}|e^{*}_{\eta_{p}}\bP_{I_{p}}(x)|\leq \frac{14}{m_{p_1}}.
\end{equation}
\end{corollary}
\begin{proof}
From Lemma  \ref{basisest} we get
\begin{align*}
\sum_{p=1}^{n_{i}}|e^{*}_{\eta_{p}}\bP_{I_{p}}(x)|&\leq
\sum_{p:l_p<j}|e^{*}_{\eta_{p}}\bP_{I_{p}}(x)|+
\sum_{p:l_p>j}|e^{*}_{\eta_{p}}\bP_{I_{p}}(x)|
\\
&\leq 
\sum_{p:l_p<j}\frac{7}{m_{p}}+
\sum_{p:l_p>j}(\frac{1}{n_{j}}+\frac{2m_{j}}{m_{p}})
\\
&\leq 
\sum_{p:l_p<j}\frac{7}{m_{p}}+
\frac{n_i}{n_{j}}+\sum_{p:l_p>j}\frac{2}{m_{p-1}}\leq \frac{14}{m_{p_1}}.
\end{align*}

\end{proof}

\section{The space $\mathfrak{X}_{Kus}$}

In this section we define the space $\mathfrak{X}_{Kus}$. We shall need the following notion from \cite{AMo}.

\begin{definition} Let $\mathfrak{X}$ be a BD-$\mathscr{L}_\infty$-subspace of $\ell_\infty(\tilde{\Gamma})$. A subset $\Gamma$ of $\tilde{\Gamma}$ is called self-determined provided $\langle\bar{d}_\gamma^{*}: \gamma\in\Gamma\rangle = \langle e_\gamma^{*}: \gamma\in\Gamma\rangle$, where $(\bar{d}_\gamma^{*})_{\gamma\in\Gamma}$ denotes the biorthogonal sequence to the basis $(\bar{d}_\gamma)_{\gamma\in\Gamma}$ and for $\gamma\in\Gamma$, $e_\gamma^{*}$ denotes the element $e_{\gamma}$ of $\ell_1(\tilde{\Gamma})$ restricted to $\mathfrak{X}$.
 \end{definition}

Now we proceed to the choice of a self-determined subset $\Gamma$ of $\bGamma$ which will determine the space $\mathfrak{X}_{Kus}$. This set will consist of regular and special nodes.

We introduce first the notion which will describe the ''freedom'' in choosing special nodes.

For any $\gamma\in\bGamma$ we write $\rank(\bdp (e^{*}_\gamma))=\{\rank \xi_i, i\in A\}$, where $\bdp (e^{*}_\gamma)=\sum_{i\in A}d^{*}_{\xi_i}$. 
\begin{definition}\label{tree-interval}
We say that the functionals $e^{*}_\gamma,e^{*}_{\tilde{\gamma}}$, $\gamma, \tilde{\gamma}\in\bar{\Gamma},$ have compatible tree-analyses if 
 \begin{enumerate}
 \item[(CT1)] $e^{*}_\gamma,e^{*}_{\tilde{\gamma}}$ have tree-analyses $(I_t,\varepsilon_t,\eta_t)_{t\in\T},\ (I_t,\tilde{\varepsilon}_t,\tilde{\eta}_t)_{t\in\T}$ respectively,
 \item[(CT2)] $\w(\eta_t)=\w(\tilde{\eta}_t)$  for any $t\in\T$, 
 \item[(CT3)] $\mtsupp e^{*}_{\eta_t}=\mtsupp e^{*}_{\tilde{\eta}_t}$ for any $t\in\T$,
 \item[(CT4)] $\rank(\eta_t)=\rank(\tilde{\eta}_t)$ for any $t\in\T$,
 \item[(CT5)] $\rank(\bdp (e^{*}_{\eta_t}))=\rank(\bdp (e^{*}_{\tilde{\eta}_t}))$ for any $t\in\T$.
 \end{enumerate}
\end{definition}

For every $\gamma=(q+1,\xi, m_{k},\epsilon,I,\esga[\eta])\in \bGamma$ and $x\in \mathfrak{X}_{\bGamma}$ we set
\begin{equation}
 \lambda_{\gamma,x}=\begin{cases} \epsilon e^*_\eta (x) \,\,\,&\textrm{if}\,\,\esga[\eta](x)\ne 0\\
\epsilon n_{k}^{-2} &\textrm{otherwise}.
\end{cases}
\end{equation}
Notice that in the above formula we do not use the projection $P_I$, which in particular yields that $|\lambda_{\gamma,x}|\leq 1$ for $x$ with $\norm[x]\leq 1$. On the other hand, for any $x$ with $\rng (x)\subset I$ we have $e^*_\eta (x)=e^*_\eta P_I(x)$ and we shall use the above notion in such context.

\begin{definition}[The tree of the special sequences]
We denote by $\mathcal{Q}$ the set of all finite sequences of pairs $\{(\zeta_1, \bar{x}_1),\ldots,(\zeta_k, \bar{x}_k)\}$ satisfying the following:
\begin{itemize}
\item[(i)] $\zeta_i\in\bar{\Ga}$ with $\rank(\zeta_i)=q_{i}\geq\min\ran_{FDD} \bar{x}_{i}$ for $i=1,\ldots,k$, 
\item[(ii)]  $(\bar{x}_{1}, \dots,\bar{x}_k)$ are vectors with rational coefficients with respect to the basis $(\brd_{\gamma})_{\gamma\in\bGamma}$, successive with respect to the FDD $(\bar{M}_q)_q$.
\end{itemize}
\end{definition}
We choose a one-to-one function $\sigma:\mathcal{Q}\rightarrow\N$, called the coding function, so that 
\begin{equation}\label{coding growth}
\sigma\left(\left\{\left(\zeta_1, \bar{x}_1\right),\ldots,\left(\zeta_k, \bar{x}_k\right)\right\}\right) > \w(\zeta_k)^{-1}\max \supp_{FDD} \bar{x}_{k}\quad\forall \{(\zeta_1,\bar{x}_1),\ldots,(\zeta_k, \bar{x}_k)\}\in\mathcal{Q}.
\end{equation}
\begin{definition}\label{defspecial}
A finite sequence $(\zeta_i, \bar{x}_i)_{i=1}^d\in\mathcal{Q}$ is called a $j$-special sequence, $j\in\N$, if $d\leq n_{2j-1}$ and the following conditions are satisfied.
\begin{enumerate}
\item[(i)] $\zeta_{1}=(q_{1}+1,0,m_{2j-1}, I_{1},\epsilon_{1},\egs[\eta_{1}])$ and $\zeta_{i}=(q_{i}+1,\zeta_{i-1},m_{2j-1},I_{i},\epsilon_{i},\lambda_{i}\egs[\eta_{i} ] )$ for every $i\leq d$,
\item[(ii)] $\w(\eta_{1})=m_{4l-2}^{-1}<n_{2j-1}^{-2}$ and $\w(\eta_i) = m_{4\sigma((\zeta_1, \bar{x}_1),\ldots,(\zeta_{i-1}, \bar{x}_{i-1}))}^{-1}$ for $i=2,\ldots,d$.
\item[(iii)] if $i$ is odd then $\lambda_{i}=1$ and $\norm[\bar{x}_i]\leq 1$,
\item[(iv)] if $i$ is even then $\epsilon_{i}=1$, $\eta_{i}$ is chosen to satisfy
$$ 
\mt(\egs[\eta_{i}])= m_{4\sigma({(\zeta_p,\bar{x}_p)_{p=1}^{i-1}})}^{-1} \sum_{r=1}^{n_{4\sigma({(\zeta_{p},\bar{x}_p)_{p=1}^{i-1}})}}\bar{d}^{*}_{\beta_{r}},
$$
where $(\bar{d}_{\beta_r})_r$ are pairwise non-neighbours. Moreover, we let
$$
\bar{x}_{i}=\frac{m_{4\sigma((\zeta_{k},\bar{x}_{k})_{k=1}^{i-1})}}{n_{4\sigma((\zeta_{k},\bar{x}_{k})_{k=1}^{i-1})}}\sum_{r=1}^{n_{4\sigma((\zeta_{k},\bar{x}_{k})_{k=1}^{i-1})}}\bar{d}_{\beta_{r}}
$$
and $\lambda_{i}\in \Net_{1,q_{i}}$ is chosen to satisfy $\vert\lambda_{i}-\lambda_{\zeta_{i-1}, \bar{x}_{i-1}}\vert<\sfrac{1}{4n_{q_{i-1}}^{2}}$.

\end{enumerate}

We denote by $\mathcal{U}$ the tree of all special sequences, endowed with the natural ordering ``$\sqsubseteq$'' of initial segments.

Fix $\Gamma=\cup_{q}\Gamma_{q}$, $\Gamma_{q}\subset \bGamma_{q}$. A $j$-special sequence $(\zeta_1,\bar{x}_1)$, with $\zeta_1=(q+1,0,m_{2j-1}, I_{1},\epsilon,\egs[\eta])$ is called $(\Gamma, j)$-special if $\eta\in\Gamma_{q}$. 
A $j$-special sequence $(\zeta_{i},\bar{x}_{i})_{i=1}^{d}$, $d\leq n_{2j-1}$, with $\zeta_{i}=(q_{i}+1,\zeta_{i-1},m_{2j-1},I_{i},\epsilon_{i}, \lambda_{i} e^{*}_{\eta_{i}})$ is called $(\Gamma,j)$-special if $\eta_{d}\in\Gamma_{q}\setminus\Gamma_{q_{d}-1}$, $\zeta_{d-1}\in\Gamma_{q_{d-1}+1}$ and $(\zeta_{i},\bar{x}_{i})_{i=1}^{d-1}$  is a $(\Gamma_{q_{d-1}},j)$-special sequence.
\end{definition}

Now we are ready to define inductively on $q\in\N$ the families of nodes $(\Delta_q)_q$ and $(\Gamma_q)_q$ satisfying $\Delta_q\subset\bDelta_q$ and $\Gamma_q=\cup_{p=1}^q\Delta_p$ for any $q\in\N$.

Set $\Gamma_1=\Delta_1=\bDelta_{1}$. Fix $q\in\N$ and 
assume we have defined all objects up to $q$-th level. 

The set of regular nodes is defined as
\begin{align*}
\Delta_{q+1}^{reg}&=\bigcup_{j=1}^{\lfloor( q+1)/2\rfloor}
\{
(q+1, 0, m_{2j}, I, \epsilon, e^{*}_{\eta}) \in\bar{\Delta}_{q+1}: \eta\in\Gamma_q \}
\\
&
\cup\bigcup_{1\le p<q}\bigcup_{j=1}^{\lfloor (q+1)/2\rfloor}
\{(q+1, \xi, m_{2j}, I, \epsilon, e^{*}_\eta)\in \bar{\Delta}_{q+1}:      \xi\in \Delta_p,  \eta\in\Gamma_q\setminus\Gamma_p
\}
\end{align*}

Now we define the special nodes, i.e. the nodes compatible to the special sequences defined above (counterparts of special functionals in \cite{AM}).
We start with the notion of compatibility, which is defined recursively on $\age(\gamma)$.

\begin{definition}\label{defcompatible}
We say that a node $\gamma=(q+1,0,m_{2j-1},I,\epsilon,\esga[\eta])\in\bDelta_{q+1}$ is compatible with a $(\Gamma_q,j)$-special sequence $(\zeta_{1},\bar{x}_{1})$, where $\zeta_1=(q+1,0,m_{2j-1}, I, \epsilon_1, \esga[\eta_1])$, if $\eta\in\Gamma_q$ and $\eta,\eta_1$ have compatible tree-analyses. 

We say that a node $\gamma=(q+1,\xi,m_{2j-1}, I,\epsilon,\lambda\esga[\eta])\in\bDelta_{q+1}$ is compatible with a $(\Gamma_q,j)$-special sequence $(\zeta_{i},\bar{x}_{i})_{i=1}^{\age(\gamma)}$, where $\zeta_{\age(\gamma)}=(q+1,\zeta_{\age(\gamma)-1},m_{2j-1},I,\epsilon_{\age(\gamma)},\lambda_{\ageg}\esga[ \eta_{\ageg}])$, provided
\begin{enumerate}
\item $\eta,\xi\in\Gamma_q$,
 \item $\xi$ is compatible with the $(\Gamma_{\rank(\xi)},j)$-special sequence $(\zeta_{i},\bar{x}_{i})_{i=1}^{\age(\xi)}$ (recall that $\age(\gamma)=\age(\xi)+1$)
\item if $\age(\gamma)$ is odd then $\lambda=1(=\lambda_{\age(\gamma)})$ and $\eta, \eta_{\ageg}$ have compatible tree-analyses,
\item if $\age(\gamma)$ is even then $\epsilon=1$, $\eta=\eta_{\age(\gamma)}$ and $\lambda\in \Net_{1,q}$ is chosen to satisfy $\vert\lambda-\lambda_{\xi, \bar{x}_{\age(\xi)}}\vert<\sfrac{1}{4n_{\rank(\xi)}^{2}}$.
\end{enumerate}
\end{definition}
The set of special nodes is defined as
\begin{align}\label{defDeltaq}
\Delta_{q+1}^{sp}&=\bigcup_{j=1}^{\lfloor( q+1)/2\rfloor}
\{
\gamma=(q+1, 0, m_{2j-1}, I, \epsilon, e^{*}_{\eta})\in\bDelta_{q+1}: \gamma\, \textrm{is compatible with some }
\\
 &\hspace{8.5cm}(\Gamma_q,j)\textrm{-special sequence}\,\, (\zeta_{1},\bar{x}_{1}) \}
\notag
\\
&
\cup\bigcup\limits_{p=1}^{q} \bigcup\limits_{j=1}^{\lfloor (q+1)/2\rfloor} 
\{
\gamma=(q+1, \xi, m_{2j-1}, I, \epsilon, \lambda e^{*}_\eta)\in\bDelta_{q+1}: \gamma\,\, \textrm{is compatible with some } 
\notag 
\\
&\hspace{8.5cm}(\Gamma_q,j)\textrm{-special sequence}\,\, (\zeta_{i},\bar{x}_{i})_{i=1}^{\ageg}
\notag
\}.
\end{align}

Finally we set 
\begin{align*}
\Delta_{q+1}=\Delta_{q+1}^{reg}\cup\Delta_{q+1}^{sp} \text{ and }\Gamma_{q+1}=\Gamma_q\cup\Delta_{q+1}.
\end{align*}
Obviously $\Delta_q\subset\bDelta_q$ for any $q\in\N$. We set $\Gamma=\cup_q\Gamma_q$. Following \cite{AMo} we denote by $R$ the restriction on $\mathfrak{X}_{\bGamma}$ of the restriction operator 
$\ell_\infty(\bGamma)\to\ell_\infty(\Gamma)$ and for any $q\in\N$ we let $i_q:\ell_\infty(\Gamma_q)\to \ell_\infty(\Gamma)$ be defined by $i_q(x)=R(\bar{i}_q(x))$ for any $x$. Given any $q\in\N$ we let $M_q=i_{\max \Gamma_q}[\ell_\infty(\Gamma_q)]$. 

\begin{proposition}
The set $\Gamma$ is a self-determined subset of $\bGamma$, hence it defines a BD-$\mathscr{L}_\infty$-space $\mathfrak{X}_{(\Gamma_{q},i_{q})_{q}}$. 

Moreover, the restriction $R:\mathfrak{X}_{\bGamma}\to\mathfrak{X}_{(\Gamma_q,i_q)_q}$ is a well-defined operator of norm at most 1 inducing the isomorphism between $\mathfrak{X}_{(\Gamma_{q},i_{q})_{q}}$ and $\mathfrak{X}_{\bGamma}/Y$, where $Y=\overline{\langle\bar{d}_\gamma: \gamma\in\bGamma\setminus\Gamma\rangle}$. 
\end{proposition}
\begin{proof}
According to Proposition 1.5~\cite{AMo} it is enough to show that for every $\gamma\in\Delta_{q+1}$ the following holds
$$
\bar{c}^*_\gamma\in \{\esga\circ P_{E}: \gamma\in\Gamma_{q}, E\subset\N\cup\{0\}\}
$$
This follows readily from the definition of $\bar{c}^*_\gamma$, see \eqref{cgamma1}, using that $\bar{d}^*_\gamma=\esga\circ P_{\{\rank(\gamma)\}}$.

The second part of Proposition follows by Proposition 1.9 \cite{AMo}.
\end{proof}
\begin{definition}
 We let $\mathfrak{X}_{Kus}=\mathfrak{X}_{(\Gamma_q,i_q)_q}$.
\end{definition}
In the sequel we shall use the casual notation, $c^{*}_{\gamma}, d^{*}_{\gamma},d_{\gamma}$ etc for the objects in the space  $X_{Kus}$. We shall use also notation $P_I$ for the projections onto $\langle d_\gamma: \ \gamma\in I\rangle$,  notice here that we can consider $I$ to be an interval in $\Gamma$ instead of $\bGamma$. 
Henceforth, by $(\gamma_n)_n$ we shall denote the enumeration of $\Gamma$ instead of the one of $\overline \Gamma$.

\begin{remark}
Notice that all the results from Section 3 are valid also for the basis $(d_\gamma)_{\gamma\in\Gamma}$ of the space $\mathfrak{X}_{Kus}$, as $d_\gamma=R\bar{d}_\gamma$, $R^{*}e^{*}_{\gamma}=e^{*}_{\gamma}$, $\gamma\in\Gamma$ by Remark 1.11 \cite{AMo} and $\norm[R]=1$. 
\end{remark}

By Proposition 1.13 \cite{AMo} we can use the analysis of nodes introduced in Section 2 in the space $\mathfrak{X}_{Kus}$. We write now the precise form of each $e^{*}_\gamma$ depending on the type of the node $\gamma\in\Gamma$.

From now on, unless specified otherwise, each node $\gamma$ shall be assumed to be in $\Gamma$.
\begin{remark}
Let a node $\gamma$ have evaluation analysis $(I_i,\epsilon_i, e_{\eta_i}^{*},\xi_i)_{i=1}^{a}$. Then 
\begin{enumerate}
\item if $\w(\gamma)=m_{2j}^{-1}$ then 
\[
e_\gamma^{*}=\sum_{i=1}^{a} d_{\xi_i}^{*}+\frac{1}{m_{2j}}\sum_{i=1}^{a}\epsilon_i P^{*}_{I_i}e_{\eta_i}^{*} ,
\] 
\item if $\w(\gamma)=m_{2j-1}^{-1}$, then 
\[
 e_\gamma^{*}=\sum_{i=1}^{a} d_{\xi_i}^{*}+\frac{1}{m_{2j-1}}\left(\sum_{i=1}^{\lfloor a/2\rfloor}(\epsilon_{2i-1}P^{*}_{I_{2i-1}}e_{\eta_{2i-1}}^{*} +\lambda_{2i}P^{*}_{I_{2i}}e^{*}_{\eta_{2i}} )+ [\epsilon_a P^{*}_{I_{a}}e^{*}_{\eta_a} ]\right),
\] 
where the last term in the square brackets appears if $a\in 2\N+1$, and with each $e^{*}_{\eta_{2i}}$ having the mt-part of the following form
$$
\mt(e^{*}_{\eta_{2i}})=\w(\eta_{2i})\sum_k P^{*}_{\Delta_{\rank(\beta_{k,i})}} e^{*}_{\beta_{k,i}} =\w(\eta_{2i})\sum_k d^{*}_{\beta_{k,i}}.
$$
\end{enumerate}
\end{remark}

Now we make some comments concerning the possible modification of the mt-part of a functional.

\begin{remark}\label{mt-part}
\begin{enumerate}
 
\item Fix $(\eta_s)_{s=1,\dots,a}$, with $a\leq n_{2j}$, $2j\leq q_{1}$, $\eta_s\in\Gamma_{q_s}\setminus \Gamma_{p_s}$, $s=1,\dots,a$, $p_1<q_1<\dots <p_{a}<q_{a}$, and $(I_s)_{s=1}^a$ with $I_s\subset\Gamma_{q_s}\setminus\Gamma_{p_s}$, $P^{*}_{I_{s}}e^{*}_{\eta_{s}}\ne 0$ and $(\epsilon_s)_{s=1,\dots, a}\subset \{\pm 1\}$.

Then the formulas $\xi_{1}=(q_{1}+1,0,m_{2j},I_1, \epsilon_{1}, e^{*}_{\eta_{1}})$ and $\xi_{s}=(q_{s}+1,\xi_{s-1},m_{2j},I_{s}, \epsilon_{s},e^{*}_{\eta_{s}})$ for any $s\leq a$ give well-defined regular nodes.

It follows that for any functional $e^{*}_\gamma$ given by a regular node $\gamma$ with 
$$
\mt(e^{*}_\gamma)=\frac{1}{m_{2j}}\sum_{i=1}^a\epsilon_ie^{*}_{\eta_i}P_{I_i}
$$
and any $(\tilde{\epsilon}_i)_{i\leq a}\subset\{\pm 1\}$ and any $(\tilde{\eta}_i)_{i\leq a}$ with $\rank (\tilde{\eta}_i)=\rank (\eta_i)$ and $P^{*}_{I_{i}}e^{*}_{\tilde{\eta}_{i}}\ne 0$ there is a regular node $\tilde{\gamma}$ with 
$$
\mt(e^{*}_{\tilde{\gamma}})=\frac{1}{m_{2j}}\sum_{i=1}^a\tilde{\epsilon}_ie^{*}_ { \tilde{\eta}_i}P_{I_i} \ \ \text{ and }\ \ \rank(\tilde{\gamma})=\rank(\gamma).
$$

\item Take a functional $e^{*}_{\gamma}$ where $\gamma$ is compatible with a $j$-special sequence, with 
\[
\mt(e_\gamma^{*})=\frac{1}{m_{2j-1}}\left(\sum_{i=1}^{\lfloor a/2\rfloor}(\epsilon_{2i-1}e_{\eta_{2-i}}^{*} P_{I_{2i-1}}+\lambda_{2i}e^{*}_{\eta_{2i}} P_{I_{2i}}) + [\epsilon_a e^{*}_{\eta_a} P_{I_a}]\right),
\] 
evaluation analysis $(I_i,\epsilon_i, e_{\eta_i}^{*},\xi_i)_{i=1}^{a}$ and weight $\w(\gamma)=m_{2j-1}^{-1}$. Let $(\tilde{\epsilon}_i)_{i=1}^a$ and $(\tilde{\eta}_i)_{i=1}^a$ satisfy the following: 
 \begin{enumerate}
\item[(i)]if $i$ is even then $\tilde{\epsilon}_i=1$, $\tilde{\eta}_i=\eta_i$,
\item[(ii)]if $i$ is odd then $\tilde{\eta_{i}}$ has compatible tree-analysis with $\eta_{i}$.
\end{enumerate}
Then the formulas $\tilde{\xi}_1=(q_1+1,0,m_{2j-1},I_1,\tilde{\epsilon}_1,\tilde{\lambda}_1e^*_{\tilde{\eta}_1})$ and $\tilde{\xi}_{i}=(q_i+1,\tilde{\xi}_{i-1},m_{2j-1},I_{i},\tilde{\epsilon}_{i}, \tilde{\lambda}_ie^*_{\tilde {\eta}_{i}})$, $i\leq a$, give well-defined special nodes. Indeed, it follows from directly  applying the definition of a special node that there exists a node $e^*_{\tilde \gamma}$ with 
\[ 
\mt(e_{\tilde{\gamma}}^{*})=\frac{1}{m_{2j-1}}\left(\sum_{i=1}^{\lfloor a/2\rfloor}(\tilde{\epsilon}_{2i-1}e_{\tilde{\eta}_{2-i}}^{*}P_{I_{2i-1}}+ \tilde{\lambda}_{2i}e^{*}_{\tilde{\eta}_{2i}}P_{I_{2i}}) + [\tilde\epsilon_a e^{*}_{\tilde\eta_a} P_{I_a}]\right) 
\]
where the $\tilde{\lambda}_{2i}$ are chosen to satisfy Definition \ref{defcompatible} (4), that is comparable with $\esga$ and hence it is a special node with the same rank as $\esga$.
\end{enumerate}
\end{remark}

\begin{remark}
Notice that by the definition of $\Gamma$, for any $\gamma\in\Gamma$ with a tree-analysis $(I_t,\epsilon_t,\eta_t)_{t\in\T}$ we have $\eta_t\in\Gamma$ and $\supp\bdp(\eta_t)\subset\Gamma$ for any $t\in\T$. 
\end{remark}

\section{Rapidly Increasing Sequences}
From now on we shall work in the space $\mathfrak{X}_{Kus}$. In this section we introduce the basic canonical tool, i.e. Rapidly Increasing Sequences and state their properties, in particular the fundamental property of Bourgain-Delbaen spaces in the Argyros-Haydon setting that allows to pass from strictly singular operators to compact ones. As the proofs of all the results stated here follows directly the reasoning of \cite{AH}, we do not present them here. 

Recall that skipped block sequences are defined with respect to the FDD 
$(M_{q})_{q\in\N}$.

\begin{definition}
Let $I$ be an interval in $\N$ and $(x_{k})_{k\in I}\subset\mathfrak{X}_{Kus}$ be a skipped block sequence. We shall say that $(x_{k})_{k\in I}$ is a Rapidly Increasing Sequence with constant $C>0$ ($C$-RIS) if there exists an increasing sequence $(j_{k})_{n\in I}\subset\N$ such that
 \begin{enumerate}
 \item $\norm[x_{k}]\leq C$ for all $k\in I$,
 \item $\rng_{FDD} x_{k}<j_{k+1}$,
 \item $|x_{k}(\gamma)|\leq Cm_{i}^{-1}$ for all $\gamma$ with $\w(\gamma)=m_{i}^{-1}$ and $i<j_{k}$.
 \end{enumerate}
\end{definition}

\begin{lemma}[Proposition 5.6 \cite{AH}]\label{p56}
Let $(x_{k})_{k=1}^{n_{j_{0}}}$ be a $C$-RIS and $s\in\N$.

a) If $\gamma\in\Gamma$ and $\w(\gamma)=m_{i}^{-1}$ then
\begin{equation}\label{bw}
\abs{e^{*}_{\gamma}P_{(s,+\infty)}\left(\frac{1}{n_{j_{0}}}\sum_{k=1}^{n_{j_{0}}
} x_{k}\right) }
 \leq
 \begin{cases}
16Cm_{i}^{-1}m_{j_{0}}^{-1}\,\,\,&\textrm{if}\,\,i<j_{0}\\
5Cn_{j_{0}}^{-1}+6Cm_{i}^{-1} &\textrm{if}\,\, i\geq j_{0}\,.
\end{cases}
\end{equation}
In particular for $i>j_{0}$ we have
\begin{equation}\label{p56a}
 \abs{e^{*}_{\gamma}\left(\frac{1}{n_{j_{0}}}\sum_{k=1}^{n_{j_{0}}}x_{k}\right)} \leq 10Cm_{j_{0}}^{-2}
\end{equation}
and also
\begin{equation}\label{p56b}
\norm[ \frac{1}{n_{j_{0}}} \sum_{k=1}^{n_{j_{0}}}x_{k}] \leq 10Cm_{j_{0}}^{-1}.
\end{equation}

b) If $\lambda_{k}$, $1\leq k\leq n_{j_0}$ are scalars with $\vert\lambda_{k}\vert\leq 1$, satisfying the property
$$
\vert \esga\left(\sum_{k\in J}\lambda_{k}x_{k}\right)\vert\leq C\max_{k\in J}\vert\lambda_{k}\vert
$$
for every $\gamma\in\Gamma$ with $\w(\gamma)=m_{j_{0}}^{-1}$ and every interval $J\subset\{1,\dots,n_{j_0}\}$ then we have
$$
\norm[\frac{1}{n_{j_0}}\sum_{k=1}^{n_{j_{0}}}\lambda_{k}x_{k}]\leq \frac{10C}{m_{j_{0}}^{2}}.
$$
\end{lemma}

The following result is proved in a manner similar to how Lemma \ref{basisest} is proved.

\begin{corollary}\label{diffw} 
Let $i<j\in\N$, $(x_{k})_{k=1}^{n_{j}}$ be a $C$-RIS, $x=\frac{m_j}{n_j}\sum_{k=1}^{n_j}x_k$ and $(\esga[\eta_{p}])_{p=1}^{n_{i}}$ be nodes such that $\w(\esga[\eta_{p}])=m_{l_p}^{-1}$ and $m_{l_{p}}\ne m_{j}$, $m_{l_{p}}<m_{l_{p+1}}$ for all $p\leq n_{i}$. Then for every choice of intervals $I_{p}$, $p\leq n_{i}$, we have
\begin{equation}
\label{ediffw} 
\sum_{p=1}^{n_{i}}|e^{*}_{\eta_{p}}(P_{I_{p}}x)|\leq 64C/m_{p_{1}}.
\end{equation}
\end{corollary}

\begin{lemma}[Corollary 8.5 \cite{AH}]\label{ris-exists}
For every block subspace $Y\subset \mathfrak{X}_{Kus}$, $C>2$ and every interval $J\subset \N$ there exists a normalized $C$-RIS $(x_k)_{k\in J}$ in $Y$. Moreover, for any $\varepsilon>0$ and $C>2$ the sequence $(x_k)_{k\in J}$ can be chosen to satisfy $|d^{*}_\gamma(x_k)|<\varepsilon$ for any $k\in J$ and $\gamma\in \Gamma$.

\end{lemma}

Notice that if $x\in \oplus_{n=1}^q M_n$ with $q$ minimal then there exists a unique $u\in \ell_\infty(\Gamma_q)$ such that $i_q(u)=x$. The local support of $x$ is defined to be the set $\{\gamma\in\Gamma_q\mid u(\gamma)\neq 0\}$. Next 
results are again quoted from \cite{AH}.

\begin{lemma}[Lemma 5.8 \cite{AH}]\label{LocSuppWeight}
Let $\gamma\in \Gamma$ be of weight $m_h^{-1}$ and assume that $\w(\xi) \ne m_{h}^{-1}$ for all $\xi$ in the local support of $x$. Then $|x(\gamma)|\le 4m_h^{-1}\|x\|.$
\end{lemma}

We recall the two classes of block sequences, characterised by the weights of the elements of the local support.

\begin{definition}[Definition 5.9 \cite{AH}]
We say that a block sequence $(x_k)_{k\in \N}$ in $\mathfrak{X}_{Kus}$ has bounded local weight if there exists some $j_1$ such that $\w(\gamma)\geq m_{j_1}^{-1}$ for all $\gamma$ in the local support of $x_k$, and all values of $k$. 

We say that a block sequence $(x_k)_{k\in\N}$ in $\mathfrak{X}_{Kus}$ has rapidly increasing local weight if, for each $k$ and each $\gamma$ in the local support of $x_{k+1}$, we have $\w (\gamma)<m_{i_k}^{-1}$ where $i_k=\max\ran_{FDD} x_k$.
\end{definition}

\begin{proposition}[Prop. 5.10 \cite{AH}]\label{LocWeightRIS}
Let $(x_k)_{k\in \N}\subset\mathfrak{X}_{Kus}$ be a bounded block sequence. If either $(x_k)$ has bounded local weight, or $(x_k)$ has rapidly increasing local weight, then the sequence $(x_k)$ is a RIS.
\end{proposition}

\begin{corollary}[Prop. 5.11 \cite{AH}]\label{fundamental}
Let $Y$ be any Banach space and $T:\mathfrak{X}_{Kus}\to Y$ be a bounded linear operator. If $\norm[Tx_{k}]\to 0$ for every RIS $(x_{k})_{k}$ in $\mathfrak{X}_{Kus}$ then $\norm[Tx_{k}]\to 0$ for every bounded block sequence $(x_k)$ in $\mathfrak{X}_{Kus}$.
\end{corollary}

\begin{corollary}[Prop. 5.12 \cite{AH}]\label{shrinking}
 The basis $(d^{*}_{\gamma_n})_n$ is shrinking. It follows that the dual space to $\mathfrak{X}_{Kus}$ is isomorphic to $\ell_1(\Gamma)$. 
\end{corollary}

\section{Dependent sequences}
In this section we introduce the classical tools in the study of spaces defined with the use of saturated norms. 
\begin{lemma}\label{normnode}
a) Let $j\in\N$ and $k\leq n_{2j}$. Let also $(x_{k})_{k}\subset\mathfrak{X}_{Kus}$ be a normalized skipped block sequence such that $\rngf(x_{k})=(p_{k-1}, p_{k}]$ for some strictly increasing $(p_k)$ with $p_1\geq 2j-1$. Then there exists a node $\gamma\in\Gamma$ such that
$$
e^{*}_{\gamma}=\sum_{k=1}^{n_{2j}}d^{*}_{\xi_{k}}+m_{2j}^{-1}\sum_{k=1}^{n_{2j}}\varepsilon_ke^{*}_{\eta_{k}} P_{I_{k}}
$$
with the following properties
\begin{enumerate}
\item[(i)] $\rank (\xi_{k})=p_{k}+1$ for each $k$, 

\item[(ii)] $\varepsilon_ke^{*}_{\eta_{k}} P_{I_{k}}(x_{k})\geq \sfrac{1}{2}$ and $\eta_{k}\in\Gamma_{p_{k}}\setminus \Gamma_{p_{k-1}}$ for each $k$,

\item[(iii)] $e^{*}_{\gamma}(\sum_{k=1}^{n_{2j}}x_{k})\geq \frac{n_{2j}}{2m_{2j} } $.
 
\end{enumerate}

\noindent b) Let $(d_{\xi_{i}})_{i=1}^{n_{2j}}$ be a finite subsequence of the basis such that $\rank(\xi_{i})+1<\rank(\xi_{i+1})$ for every $i$ and $\rank(\xi_1)\geq 2j-1$.

Then the node
\begin{equation}
e^{*}_{\xi}=\sum_{i=1}^{n_{2j}}d^{*}_{\zeta_{i}}+m_{2j}^{-1}\sum_{i=1}^{n_{2j}}d^{*}_{\xi_{i}}
\end{equation}
with $\rank(\zeta_{i})=\rank(\xi_{i})+1$ is a regular node and $e^{*}_{\xi}(\sum_{i=1}^{n_{2j}}d_{\xi})=\frac{n_{2j}}{m_{2j}}$.
\end{lemma}
\begin{proof} a) (see\cite{AH}, Proposition 4.8) Let $x_{k}=i_{k}(u_{k})$ where $u_{k}\in\Gamma_{p_{k}}\setminus 
\Gamma_{p_{k-1}}$ is the restriction of $x_{k}$ on $\Gamma_{p_{k}}$. Since
$$
2\norm[u_{k}]\geq \norm[i_{p_{k}}(u_{k})]=\norm[x_{k}]=1
$$
we can choose $\eta_{k}\in\Gamma_{p_{k}}\setminus \Gamma_{p_{k-1}}$ such that $\abs{e^{*}_{\eta_{k}}(u_{k})}\geq \sfrac{1}{2}$. Setting $I_{k}=\rng_{FDD}(x_{k})=\cup_{i=p_{k-1}+1}^{p_{k}}\Delta_{i}$, choose $\varepsilon_k\in\{-1,1\}$ such that
\begin{equation}
 \label{eq:5}
\abs{e^{*}_{\eta_{k}}P_{I_{k}}(x_{k})}=\varepsilon_ke^{*}_{\eta_{k}}P_{I_{k}}(x_{k})=\varepsilon_ke^{*}_{\eta_{k}}(u_{k})\geq \sfrac{1}{2}.
 \end{equation}

The nodes $\gamma_{k}=(p_{k}+1,\gamma_{k-1},m_{2j},I_{k},\varepsilon_k,e^{*}_{\eta_{k}})$, $\gamma_{0}=0$, $k=1,\dots,n_{2j}$ give the node $\gamma=\gamma_{n_{2j}}$ with the properties (i)-(iii).

\medskip

b) Take the nodes $\zeta_{i}=(\rank(\xi_{i})+1,\zeta_{i-1}, m_{2j}, I_{i},1,e^{*}_{\xi_{i}})$, $\zeta_{0}=0$, where $I_{i}=\Delta_{\rank(\xi_{i})}$.
\end{proof}
\begin{definition}\label{dep-def}
Fix $j\in\N$, $C\geq 1$ with $n_{2j-1}\geq 200 C$ and let $(\gamma_{k},\bar{x}_{k})_{k=1}^{d} $ be a $(\Gamma,j)$-special sequence.

A sequence $(\gamma_{k},x_{k})_{k=1}^{d}$, $d\leq n_{2j-1}$, with $x_k\in\mathfrak{X}_{Kus}$ and $\gamma_{k}=(q_{k}+1,\gamma_{k-1},m_{2j-1}, I_{k},1,\esga[\eta_{k}])$ for each $k$, where $\gamma_{0}=0$, $q_1\geq 4j_1-2$, $2^{-q_1}\leq 1/4n^2_{2j-1}$, is called a $j$-dependent sequence with a constant $C$ of length $d$ with respect to $(\gamma_{k},\bar{x}_{k})_{k=1}^{d}$ if the following conditions are satisfied.
 \begin{enumerate}
 \item if $k$ is even then $x_k=R\bar{x}_k$, $\rng (x_k)= I_k$,
\item if $k$ is odd then
${\displaystyle 
x_{k}=\frac{c_{k}m_{l_k}}{n_{l_{k}}}\sum_{l=1}^{n_{l_{k}}}x_{k,l}}$, 
where $(x_{k,l})_{l}$ is a normalized skipped block sequence which is a $C$-RIS of length $n_{l_k}$, $m_{l_{k}}=\w(\eta_{k})$, $m_{l_1}\geq n^2_{2j-1}$,  $\norm[x_{k}]=\sfrac{1}{2}$, $\rng (x_k)\subset I_k$ and $e^{*}_{\eta_{k}}(x_{k})\geq \sfrac{1}{40C}$,

\item 
$\vert\esga(\bar{x}_{k})-\esga(x_{k})\vert<\sfrac{1}{4n_{q_{k}}^{2}}$ for every $\gamma\in\Gamma$ and every $k$,

\item $(\gamma_{k},x_{k})_{k=1}^{d-1}$ is $j$-dependent of length $d-1$ with respect to the $(\Gamma,j)$-special sequence $(\gamma_{k},\bar{x}_{k})_{k=1}^{d-1}$.
 \end{enumerate}
 Moreover, we say that a sequence $(\gamma_{k},x_{k})_{k=1}^{d}$ is a $j$-dependent sequence of length $d$, if it is $j$-dependent with respect to some $(\Gamma,j)$-special sequence.
\end{definition}

\begin{remark}\label{rem6.3}
Take $(x_{k,l})_l$ as in (2) of Definition \ref{dep-def} with $\max\rng_{FDD}(x_k)\geq 2l_k-1$ for each $k\in\N$. Then Lemmas~\ref{p56}a) and \ref{normnode}a) yield that there is a node  $\eta_{k}\in\Gamma$ such that  
\begin{equation}\label{cbound}
\frac{1}{2}\leq e^{*}_{\eta_{k}}(\frac{m_{l_k}}{n_{l_{k}}}\sum_{l=1}^{n_{l_{k}}}x_{k,l})
\leq \norm[\frac{m_{l_k}}{n_{l_{k}}}\sum_{l=1}^{n_{l_{k}}}x_{k,l}]\leq 10C.
\end{equation}
Therefore $c_{k}$ in Definition \ref{dep-def} satisfies $\sfrac{1}{20C}\leq c_{k} \leq 1$. 

Moreover, the last condition in the property (2) of  Definition \ref{dep-def}, i.e. $e^{*}_{\eta_{k}}(x_{k})\geq \sfrac{1}{40C}$, follows from \eqref{cbound} using the lower bound of $c_{k}$.
\end{remark}

\begin{lemma} \label{dep1}
 Let $(z_{k})_{k}$ be a normalized block sequence in $\mathfrak{X}_{Kus}$ and $(d_{\xi_{n}})_{n\in M}$ be a subsequence of the basis. Then for every $j\in\N$ there exists a $j$-dependent sequence of length $n_{2j-1}$, 
$(\gamma_{i},x_{i})_{i\leq n_{2j-1}}$, such that $x_{2i-1}\in \langle z_{k}: k\in\N\rangle$ and $x_{2i}\in \langle d_{\xi_{n}}: n\in M\rangle$.
 \end{lemma}
 
\begin{proof}
 Passing to a further subsequence we may assume that 
 \begin{equation}
 \mbox{ $d_{\xi_{n}}$ are pairwise non-neighbours and $\rank (\xi_{n})+1<\rank (\xi_{n+1})$.}
 \end{equation}

Let $j_{1}$ be such that $m_{4j_{1}-2}>n_{2j-1}^{2}$ and choose $q_1$ big enough to guarantee that $4j_{1}-2<q_1$ and $2^{-q_1}\leq \sfrac{1}{4n^2_{2j-1}}$.

Let $(x_{1,k})_{k=1}^{n_{4j_{1}-2}}$ be a normalized skipped block sequence of $\langle z_{l}: l\geq q\rangle$ which is a $C$-RIS. 
 Setting
$$
x_{1}=\frac{c_1m_{4j_{1}-2}}{n_{4j_{1}-2}}\sum_{k=1}^{n_{4j_{1}-2}}x_{1,k}
\qquad
\textrm{with}\,\,\, \norm[x_{1}]=\sfrac{1}{2}$$
from Remark~\ref{rem6.3} we get $\sfrac{1}{20C}\leq c_{1}\leq 2$ and that there exists  a node  $\eta_{1}\in\Gamma$  with $\w(\eta_{1})=m_{4j_{1}-2}^{-1}$   such that
$$e^{*}_{\eta_{1}}P_{I_{1}}(x_{1})\geq
\frac{1}{40C},
$$
where $I_{1}=\bigcup\{\Delta_{p}: p\in\rng_{FDD}(x_{1})\}$. 

Using that $R$ is a quotient operator of norm 1 take a block $\bar{y}_{1}\in\mathfrak{X}_{\bGamma}$ such that $x_{1}=R(\bar{y}_{1})$ and $\norm[\bar y_{1}]\leq 1$. Then choose a vector $\bar{x}_{1}$ with rational coefficients in the unit ball of $\langle \bar{d}_\gamma: \gamma\in\bGamma_{q_1}\rangle$ such that $\norm[\bar{x}_{1}-\bar{y}_{1}]_{\mathfrak{X}_{\bGamma}}\leq \sfrac{1}{4n_{q_{1}}^{2}}$.

Note that $R(\bar{x}_{1})=R(\bar{x}_{1}-\bar{y}_{1})+R(\bar{y}_{1})=R(\bar{x}_{1}-\bar{y}_{1})+x_{1}$ and 
hence for every $\gamma\in\Gamma$, 

\begin{equation}\label{xapprox}
\vert \esga (\bar{x}_{1})-\esga(x_{1})\vert=
\vert\esga R(\bar{x}_{1})-\esga R(x_{1})\vert\leq
\norm[\esga\circ R]\norm[\bar{x}_{1}-\bar{y}_{1}]_{\mathfrak{X}_{\bGamma}}\leq 
\sfrac{1}{4n_{q_{1}}^{2}}.
\end{equation}

We take $\gamma_{1}$ to be the node
$$
\gamma_{1}=(q_{1}+1, 0, m_{2j-1}^{-1}, I_{1},1,e^{*}_{\eta_{1}}).
$$
From the above we get that $(\gamma_{1},x_{1})$ is a $j$-dependent couple of length $1$ with respect to the $(\Gamma,j)$-special sequence $(\gamma_{1},\bar{x}_{1})$.

Set $j_{2}=\sigma(\gamma_{1},\bar{x}_{1})$ and choose $x_{2}, e^{*}_{\eta_{2}}$ such that
$$
x_{2}=m_{4j_{2}}n_{4j_{2}}^{-1}\sum_{k\in F_{2}}d_{\xi_{2,k}}\in X_{Kus}\qquad\textrm{and}\qquad mt(\esga[\eta_{2}])=m_{4j_{2}}^{-1}\sum_{k\in F_{2}}d^{*}_{\xi_{2,k}}
$$
where $|F_2|=n_{4j_{2}}$ and $q_1+2<\min\rng_{FDD}(x_{2})$. 
Such a node exists by Lemma~\ref{normnode}(b) since $\rank (\xi_{n})+1<\rank (\xi_{n+1})$. We also take the node
$$
\gamma_{2}=(q_{2}+1,\gamma_{1},m_{2j-1}, I_{2},1,\lambda_{2}e^{*}_{\eta_{2}})\in\Gamma
$$
where $I_{2}=[p_{2},q_{2}]$ is the range of $x_{2}$ with respect to the basis and $\lambda_{2}\in\Net_{1,q_{1}}$ is chosen such that
$$
\vert\lambda_{2}-\esga[\eta_{1}](\bar{x}_{1})\vert\leq \sfrac{1}{4n_{q_{1}}^{2}}.
$$
From the above equation and \eqref{xapprox} we get
$$
\vert\lambda_{2}-\esga[\eta_{1}](x_{1})\vert\leq \frac{1}{2n_{q_{1}}^{2}}\Rightarrow
\lambda_{2}\geq \esga[\eta_{1}](x_{1})-\frac{1}{2n_{q_{1}}^{2}}\geq \frac{c_{1}}{2}-\frac{1}{2n_{q_{1}}^{2}}\geq \frac{1}{45C}.
$$
Pick $\bar{x}_2$ to be the corresponding average of $(\bar{d}_{\xi_{2,k}})_{k\in F_2}$. It follows that $x_2=R\bar{x}_2$ (recall that $d_\gamma=R\bar{d}_\gamma$ for each $\gamma\in\Gamma$) and $\bar{x}_1<\bar{x}_2$. Then we get that $(\gamma_{i},x_{i})_{i=1}^{2}$ is $j$-dependent of length 2 with respect to the $(\Gamma,j)$-special sequence
$(\gamma_{i},\bar{x}_{i})_{i=1}^{2}$.

Set $j_{3}=\sigma (\gamma_{i},\bar{x}_{i})_{i=1}^{2}$. We continue to choose $x_{3}, e^{*}_{\gamma_{3}}$, $x_{4},e^{*}_{\gamma_{4}}$ in the same way we have chosen $x_{1}, e^{*}_{\gamma_{1}}, x_{2},e^{*}_{\gamma_{2}}$ taking care that $x_{1},x_{2},x_{3},x_{4}$ is a skipped block sequence (with respect to the FDD) and repeat the procedure obtaining the desired dependent sequence. 
\end{proof}

Notice that for a dependent sequence $(\gamma_i,x_i)_{i\leq n_{2j-1}}$ with a constant $C$ we have $\norm[\frac{m_{2j-1}}{n_{2j-1}}\sum_{i=1}^{n_{2j-1}}x_{i}]\geq \sfrac{1}{45C}$. Indeed, consider the functional $e^{*}_{\zeta_{n_{2j-1}}}$ determined by the nodes $(\gamma_{i})_{i=1}^{n_{2j-1}}$,  i.e. of the form
$$
e^{*}_{\zeta_{n_{2j-1}}} =\sum_{i=1}^{n_{2j-1}}d^{*}_{\gamma_{i}}+m_{2j-1}^{-1}\sum_{i=1}^{n_{2j-1}/2}(e^{*}_{\eta_{2i-1}}P_{I_{2i-1}}+\lambda_{2i}e^{*}_{\eta_{2i}}P_{I_{2i}}),
$$
and notice that
\begin{align*}
e^{*}_{\zeta_{n_{2j-1}}}
\left(
\frac{m_{2j-1}}{n_{2j-1}}
\sum_{i=1}^{n_{2j-1}}x_{i}
\right)
&\geq 
\frac{1}{n_{2j-1}}
\left(
\sum_{i=1}^{n_{2j-1}/2}
e^{*}_{\eta_{2i-1}}P_{I_{2i-1}}(x_{2i-1})+\lambda_{2i}
e^{*}_{\eta_{2i}}(x_{2i})
\right)
\\
&\geq\frac{1}{n_{2j-1}}
\sum_{i=1}^{n_{2j-1}/2} 
\left(
\frac{c_{2i-1}}{2}+\frac{c_{2i-1}}{2}-
\frac{1}{2n_{q_{2i-1}}^{2}}
\right)
\geq \frac{1}{45C}.
\end{align*}
using that $c_{2i-1}\geq\sfrac{1}{20C}$. 

\begin{lemma}\label{dep2}
Let $(\gamma_{i},x_{i})_{i\leq n_{2j-1}}$ be a $j$-dependent sequence. Then
 $$
 \norm[\frac{1}{n_{2j-1}}\sum_{i=1}^{n_{2j-1}}(-1)^{i+1} x_{i}]\leq\frac{250}{m^{2}_{2j-1}}.
 $$
\end{lemma}

\begin{proof} Let $J$ be an interval of $\{1,\dots, n_{2j-1}\}$ and $z=\sum_{i\in J}(-1)^{i+1}x_{i}$. We shall verify the assumption (b) in Lemma~\ref{p56a} for $j_{0}=2j-1$. 

Let $(\gamma_{k},\bar{x}_{k})_{k=1}^{n_{2j-1}}$ be the special sequence associated with the dependent sequence $(\gamma_{k},x_{k})_{k=1}^{n_{2j-1}}$, $\gamma_{k}=(q_{k}+1,\gamma_{k-1},m_{2j-1},I_{k},\epsilon_{k},\lambda_{k}e^{*}_{\eta_{k}})$ for each $k$, where $\gamma_0=0$.

Consider a node $\beta$ with evaluation analysis
$$
e^{*}_{\beta}=\sum_{i=1}^{n_{2j-1}}d^{*}_{\xi_{i}}+m_{2j-1}^{-1}\sum_{i=1}^{n_{2j-1}/2}(\tilde{\epsilon}_{2i-1}e^{*}_{\tilde{\delta}_{2i}-1}P_{\tilde{I}_{2i-1}}+\tilde\lambda_{2i}e^{*}_{\tilde{\delta}_{2i}}P_{\tilde{I}_{2i}} )
$$
which is produced from a $(\Gamma, j)$-special sequence $(\zeta_{k},\bar{z}_{k})_{k\leq n_{2j-1}}$.
Let
$$
k_{0}=\min\{k\leq n_{2j-1}: (\gamma_{k},\bar{x}_{k})\ne (\zeta_{k},\bar{z}_{k})\}
$$
if such a $k$ exists. We estimate separately $|e^{*}_{\beta_{k_0-1}}(z)|$ and $|(e^{*}_{\beta}-e^{*}_{\beta_{k_0-1}})(z)|$.

We start with $|e^{*}_{\beta_{k_0}-1}(z)|$. Notice that $e^{*}_{\beta_{k_0}-1}$, if $k_0>1$, has the following evaluation analysis
$$
e^{*}_{\beta_{k_0-1}}=\sum_{i=1}^{k_0-1}d^{*}_{\xi_{i}}+m_{2j-1}^{-1}\sum_{i=1}^{\lfloor 
(k_0-1)/2\rfloor}(\tilde{\epsilon}_{2i-1}e^{*}_{\tilde{\delta}_{2i-1}}P_{I_{2i-1}}+\tilde{\lambda}_{2i}e^{*}_{\eta_{2i}}P_{I_{2i}} 
)+[\tilde{\epsilon}_{k_0-1}e^{*}_{\tilde{\delta}_{k_0-1}}P_{I_{k_0-1}}].
$$
where $e^{*}_{\tilde{\delta}_{2i-1}}$ and $e^{*}_{\eta_{2i-1}}$ have compatible tree-analyses and the last term in square brackets appears if $k_0-1$ is odd. By the definition of nodes we have $\rank 
(\xi_{i})=\rank(\gamma_{i})\in(\max\rng_{\textrm{\tiny FDD}}(x_{i}),\min\rng_{\textrm{\tiny FDD}}(x_{i+1}))$ 
for every $i<k_0$. Therefore 
\begin{equation}\label{az0}
\left(\sum_{i=1}^{k_0-1}d^{*}_{\xi_{i}}\right) 
\sum_{i}(-1)^{i+1}x_{i}=0.
\end{equation}
We partition the indices $P=\{1,2,\dots,\lfloor (k_0-1)/2\rfloor\}$ into the sets $A=\{i\in P: e^{*}_{\tilde{\delta}_{2i-1}}P_{I_{2i-1}}(\bar{x}_{2i-1})\ne 0\}$ and its complement $B$.

For every $i\in A$ from the choice of $\tilde\lambda_{2i}$, the fact that $\rng(x_{2i-1})\subset I_{2i-1}$ and (3) of Def. \ref{dep-def} we have
\begin{align}\label{b15}
\vert\tilde\lambda_{2i}-\tilde{\epsilon}_{2i-1}e^{*}_{\tilde{\delta}_{2i-1}}(\bar{x}_{2i-1})\vert &
\leq \frac{1}{4n_{2j-1}^{2}}\quad\textrm{and}\quad 
\\
\vert\notag e^{*}_{\tilde\delta_{2i-1}}(\bar{x}_{2i-1})
-e^{*}_{\tilde\delta_{2i-1}}P_{I_{2i-1}}(x_{2i-1})\vert &=\vert e^{*}_{\tilde\delta_{2i-1}}(\bar{x}_{2i-1})-e^{*}_{\tilde\delta_{2i-1}}(x_{2i-1})\vert\leq \frac{1}{4n_{2j-1}^{2}}.
\end{align}
It follows that
\begin{align}\label{az2}
\abs{\tilde{\epsilon}_{2i-1}e^{*}_{\tilde{\delta}_{2i-1}}P_{I_{2i-1}}(x_{2i-1})+\tilde\lambda_{2i}e^{*}_{\eta_{2i}}P_{I_{2i}}(-x_{2i}) }
&=\abs{\tilde{\epsilon}_{2i-1}e^{*}_{\tilde{\delta}_{2i-1}}P_{I_{2i-1}}(x_{2i-1})-\tilde\lambda_{2i}}
\leq \frac{1}{2n_{2j-1}^{2}}\quad\textrm{by}\,\,\eqref{b15}.
\end{align}
Similarly for every $i\in B$,
\begin{align}
\label{az3}
\abs{\tilde{\epsilon}_{2i-1}e^{*}_{\tilde{\delta}_{2i-1}}P_{I_{2i-1}}(x_{2i-1})+
\tilde\lambda_{2i}e^{*}_{\eta_{2i}}P_{I_{2i}}(-x_{2i})}
&=\abs{\tilde{\epsilon}_{2i-1}e^{*}_{\tilde{\delta}_{2i-1}}P_{I_{2i-1}}(x_{2i-1})-\tilde\lambda_{2i}}
\leq \frac{1}{2n_{2j-1}^2}.
\end{align}

For an interval $J=[l,m]$ using that $\norm[x_{2i-1}] = \sfrac{1}{2}$, $\norm[x_{2i}]\leq 7$ (by Lemma \ref{basisest}) and inequalities \eqref{az0}, \eqref{az3} we obtain
$$
|e^{*}_{\beta_{k_0-1}}\left(\sum_{i\in J}(-1)^{i+1}x_{i}\right)|\leq 
10.
$$
Now we proceed to estimate $|(e^{*}_{\beta}-e^{*}_{\beta_{k_0-1}})(z)|$.

Observe that as $x_{2l-1}$ is a weighted average of a normalized C-RIS of length $n_{j_{2l-1}}$ we have
\begin{equation}
 \label{eq:2} 
\abs{\left(\sum_{i=k_0}^{n_{2j-1}}d^{*}_{\xi_{i}}\right)(x_{2l-1})
}
 \leq
 3n_{2j-1}c_{2i-1}C\frac{m_{j_{2l-1}}}{n_{j_{2l-1}}}\leq 
2m_{j_{2l-1}}^{-2}<n_{2j-1}^{-3}
\end{equation}
The same inequality holds also for the averages of the basis i.e.
\begin{equation}
 \label{eq:4}
\abs{\sum_{i=k_0}^{n_{2j-1}}d^{*}_{\xi_{i}}(x_{2l})}
 \leq
 n_{2j-1}\frac{m_{j_{2l}}}{n_{j_{2l}}}\leq 
m_{j_{2l}}^{-3}<n_{2j-1}^{-3}\,\,\,\ \forall l.
\end{equation}
We shall distinguish the cases when $k_0$ is odd or even. Assume first that $k_{0}=2i_{0}-1$ for some $i_0$.

Then for every $i<i_{0}$ and every $k>k_{0}$,
$$
(\tilde{\epsilon}_{2i-1}e^{*}_{\tilde{\delta}_{2i}-1}P_{\tilde{I}_{2i-1}}+\tilde\lambda_{2i}e^{*}_{\tilde{\delta}_{2i}}P_{\tilde{I}_{2i}})(x_{k})=0.
$$
From the injectivity of $\sigma$ it follows that $\w(e^{*}_{\tilde{\delta}_{2i-1}}),\w(e^{*}_{\tilde{\delta}_{2i}})\notin\{\w(e^{*}_{\eta_{i'}})\mid i'>i_{0}\}$ for every $i>i_{0}$. Hence by Corollary~\ref{diffw}, using that $|\tilde\lambda_{2i}|\leq 1$ and $c_{k}\leq 2$, we get for every odd $k>k_{0}$ the following
\begin{align}
 \label{eq:1'}
 \abs{\sum_{i\geq i_{0}}^{n_{2j-1}/2}(\tilde{\epsilon}_{2i-1}e^{*}_{\tilde{\delta}_{2i-1}}P_{\tilde I_{2i-1}}+&\tilde\lambda_{2i}e^{*}_{\tilde{\delta}_{2i}}P_{\tilde I_{2i}})(x_k)}\leq 64c_{k}C\w(\delta_{1})\leq 128C n_{2j-1}^{-2}.
\end{align}
Also from Corollary~\ref{diffwb} we obtain for every even $k>k_{0}$ the following
\begin{align}
 \label{eq:1a'}
 \abs{\sum_{i\geq i_{0}}^{n_{2j-1}/2}(\tilde{\epsilon}_{2i-1}e^{*}_{\tilde{\delta}_{2i-1}}P_{\tilde I_{2i-1}}+&\tilde\lambda_{2i}e^{*}_{\tilde{\delta}_{2i}}P_{\tilde I_{2i}})(x_k)}
\leq
14 n_{2j-1}^{-2}.
\end{align}

For $x_{k_0}$ we also obtain the following
\begin{align}
 \label{eq:3'}
 \vert
\sum_{i\geq i_{0}}^{n_{2j-1}/2}& (\tilde{\epsilon}_{2i-1}e^{*}_{\tilde{\delta}_{2i-1}}
P_{\tilde{I}_{2i-1}}+\tilde\lambda_{2i}e^{*}_{\tilde{\delta}_{2i}}P_{\tilde{I}_{2i}})(x_ { k_0})
\vert
 \\
&\leq
\abs{ e^{*}_{\tilde{\delta}_{k_{0}}}P_{\tilde{I}_{k_{0}}} (x_{k_{0}})}+\abs{\Big(\tilde\lambda_{k_{0}+1}e^{*}_{\tilde{\delta}_{k_{0}+1}}P_{\tilde{I}_{k_{0}+1}} + \sum_{i> i_{0}}^{n_{2j-1}/2} 
(\tilde{\epsilon}_{2i-1}e^{*}_{\tilde{\delta}_{2i-1}}
P_{\tilde{I}_{2i-1}}+\tilde\lambda_{2i}e^{*}_{\tilde{\delta}_{2i}}
P_{\tilde{I}_{2i}}
)\Big)(x_{k_{0}})}
\notag\\
&\leq 4+ 128C n_{2j-1}^{-2},
 \notag
\end{align}
using that $\norm[x_{k_0}]\leq 1$ and $\norm[\esga\circ P_{I}]\leq \norm[P_{I}]\leq 4$ while for the second term we get the upper bound as in \eqref{eq:1'}.

The case where $k_{0}$ is even is similar, except that $\abs{e^{*}_{\tilde{\delta}_{k_{0}}}P_{\tilde{I}_{k_{0}}} (x_{k_{0}})}\leq 7$.

Splitting $J$ to $J_{1}=J\cap [1,i_{0}]$, $J_{2}=J\cap(i_{0},n_{2j-1})$ and considering the cases when $\min J_{1}$ is odd or even we get $|( e^{*}_{\beta}-e^{*}_{\beta_{k_0-1}})\left(\sum_{i\in J}(-1)^{i+1}x_{i}\right)|\leq 15$, using that $n_{2j+1}>200C$.
\end{proof}

The lemmas above imply the following.
\begin{proposition}\label{spheres}
 Let $M\subset\N$ be infinite and $(y_{k})_{k}\subset \mathfrak{X}_{Kus}$ be a normalized block sequence. Then
 $$
 \inf\{\norm[x-y]: x\in \langle d_{\gamma_n} :n\in M \rangle, y\in \langle y_{k}:k\in\N\rangle , \norm[x]=\norm[y]=1\}=0.
 $$
\end{proposition}

\section{Bounded operators on the space $\mathfrak{X}_{Kus}$}

In this section we show that the space $\mathfrak{X}_{Kus}$ has the scalar-plus-compact property.

\begin{proposition}\label{dist}
Let $T:\mathfrak{X}_{Kus}\to \mathfrak{X}_{Kus}$ be a bounded operator and $(d_{\gamma_{n}})_{n\in M}$ be a subsequence of the basis. Then 
$$
\lim_{M\ni n \to+\infty} \dist(Td_{\gamma_n},\R d_{\gamma_n})=0.
$$
\end{proposition}
\begin{proof} Assume that $\dist(Td_{\gamma_n},\R d_{\gamma_n})>4\delta$ for infinitely many $n\in M$ and some $\delta>0$.

By Corollary \ref{shrinking} and Lemma \ref{neighbours} passing to a further subsequence and admitting a small perturbation we may assume that
\begin{enumerate}
\item[(P1)] ${\displaystyle (Td_{\gamma_n})_{n\in M}}$ is a skipped block sequence and setting $R_{n}$ to be the minimal interval containing $\rng( Td_{\gamma_n})$ and $\{n\}$ we have 
\[
\max\rank (R_{n}) +2 <\min\rank(R_{n+1}).
\]
\item[(P2)] no two elements of $(d_{\gamma_n})_{n\in M}$ are neighbours.
\end{enumerate}

By the assumption that $\dist(Td_{\gamma_n},\R d_{\gamma_n})>4\delta$ it follows that either 
\[
\norm[P_{n-1}Td_{\gamma_n}]\geq 2\delta \,\,\,\textrm{or}\,\,\,\ \norm[(I-P_n)Td_{\gamma_n} ]\geq 2\delta 
\]
(recall that $P_{m}$ denotes the canonical projection onto $\langle d_{\gamma_i}:i\leq m\rangle$, $m\in\N$).

Passing to a further subsequence we may assume that one of the two alternatives holds for any $n\in\N$. Let 
$$
q_{n}=
\begin{cases}
\max\rank(P_{n-1}Td_{\gamma_n} )\,\,\,\, &\textrm{in the first case}
\\
\max\rank((I-P_n)Td_{\gamma_n}) &\textrm{in the second case}.
\end{cases}
$$
In the first case we take $I_{n}=[\min\rng (Td_{\gamma_n}),n-1]$. Also
$P_{n-1}Td_{\gamma_n}=i_{q_{n}}(u_{n})$ where 
$u_{n}=r_{q_{n}}(P_{n-1}Td_{\gamma_n})$ and hence we may choose 
$\epsilon_{n}\in\{-1,1\}$ and 
$\eta_{n}\in\Gamma_{q_{n}}\setminus\Gamma_{\max\rank(R_{n-1})+1}$ such that 
\begin{equation}\label{bign}
\epsilon_{n}e^{*}_{\eta_n} P_{I_{n}}(T\dg[\gamma_{n}])=
\epsilon_{n}e^{*}_{\eta_n}(P_{n-1}T\dg[\gamma_{n}])=
\epsilon_{n}e^{*}_{\eta_n}(u_{n})
\geq \delta
\end{equation}
using that $2\delta\leq\norm[i_{q_{n}}(u_{n})]\leq 2\norm[u_n]$.

In the second case we take $I_{n}=[n+1,\max\rng (Td_{\gamma_{n}})]$. Also since 
$(I-P_n)Td_{\gamma_n}=i_{q_{n}}(u_{n})$ where 
$u_{n}=r_{q_{n}}((I-P_n)Td_{\gamma_n})$ we get $\epsilon_{n}\in\{-1,1\}$, 
$\eta_{n}\in\Gamma_{q_{n}}\setminus\Gamma_{\max\ran R_{n-1}+1}$ such that 
\begin{equation}\label{bign2}
\epsilon_{n}e^{*}_{\eta_n} P_{I_{n}} (T\dg[\gamma_{n}])=
\epsilon_{n}e^{*}_{\eta_n}((I-P_n)T\dg[\gamma_{n}])=
\epsilon_{n}e^{*}_{\eta_n}(u_{n})\geq \delta.
\end{equation}

Given any $j\in\N$ we shall build a vector $y$ with $\norm[Ty]\geq \delta/28m_{2j-1}$ and $\norm[y]\leq 420/m^2_{2j-1}$ which for sufficiently big $j$ yields a contradiction.

Assume the first case holds. The second case will follow analogously. Notice that by (P1) for any $i\in\N$ and $A\subset M$ with $\# A=n_{2i}$ and $\max\rank (R_{\min A})\geq 2i-1$ there is a functional $e^{*}_\psi$ associated to a regular node of the form
$$
e^{*}_\psi=\sum_{n\in A}d^{*}_{\xi_n}+\frac{1}{m_{2i}}\sum_{n\in A}\epsilon_ne^{*}_{\eta_n}P_{I_{n}}.
$$
with $\rank(\xi_n)=\max\rank (R_n)+1$ for each $n\in A$. Let $x=m_{2i}n_{2i}^{-1}\sum_{n\in A}d_{\gamma_n}$.

It follows that 
\begin{align*}\label{eqq2}
\norm[Tx]&\geq e^*_\psi(Tx)=\left(\sum_{n\in A}d^{*}_{\xi_n}+ 
\frac{1}{m_{2i}}\sum_{n\in A}\epsilon_ne^{*}_{\eta_n}P_{I_{n}}\right)
\left(\frac{m_{2i}}{n_{2i}}\sum_{n\in A}Td_{\gamma_n}\right)
\\
&=m_{2i}n_{2i}^{-1}\sum_{n\in A}d^{*}_{\xi_n}(Td_{\gamma_n})+\frac{1}{n_{2i}}\sum_{n\in A}\epsilon_ne^{*}_{\eta_n}P_{I_{n}}(Td_{\gamma_n})
\\
&=\frac{1}{n_{2i}}\sum_{n\in A}\epsilon_ne^{*}_{\eta_n}P_{k_{n-1}}(Td_{\gamma_n})\geq \delta.
\end{align*}

Fix $j\in\N$ and choose inductively, as in Lemma \ref{dep1}, a $j$-dependent sequence $(\zeta_{i},x_{i})$, $\zeta_i=(q_i+1,\zeta_{i-1},m_{2j-1},J_i,1,\psi_i)$, $i=1,\dots, n_{2j-1}$, with $\zeta_0=0$, with respect to a $(\Gamma, j)$-special sequence $(\zeta_i,\bar{x}_i)$, so that it satisfies for any $i$ the following
$$
e^{*}_{\psi_{2i-1}}=\sum_{n\in A_i}d^{*}_{\xi_n}+\frac{1}{m_{j_{2i-1}}}\sum_{n\in A_i}\epsilon_ne^{*}_{\eta_n}P_{I_n},\qquad x_{2i-1}=\frac{c_{2i-1}m_{j_{2i-1}}}{n_{j_{2i-1}}}\sum_{n\in A_i}d_{\gamma_n},\,\,\norm[x_{2i-1}]=\sfrac{1}{2}
$$ 
with $\rank(\xi_n)=\max\rank(R_n)+1$ for each $n\in\cup_iA_i$. Lemma~\ref{basisest} yields that $\sfrac{1}{14}\leq c_{2i-1}\leq 1$. Recall that by definition each vector $\bar{x}_{2i-1}$ satisfies 
$$
\vert\esga(\bar{x}_{2i-1})-\esga(x_{2i-1})\vert\leq 4n_{q_{2i-1}}^{-2}\,\,\,\forall\gamma\in\Gamma. 
$$
For any $i$ let $J_{2i-1}=\rng (e^{*}_{\psi_{2i-1}})$. We demand also that $\supp e^{*}_{\psi_{2i}}\cap \supp x_{2k-1}=\emptyset$ for any $i,k$, thus the even parts of the chosen special functional play no role in the estimates on the weighted averages of $(x_{2i-1})$. We assume also $m_{j_1}/m_{j_1+1}\leq 1/n_{2j-1}^2$.

By the previous remark we have for each $i$ the following
\begin{equation}\label{eqq3}
\egs[\psi_{2i-1}](Tx_{2i-1})\geq \delta/14.
\end{equation}
Let
$$
y=\frac{1}{n_{2j-1}}\sum_{i=1}^{n_{2j-1}/2}x_{2i-1}
=\frac{1}{n_{2j-1}}\sum_{i=1}^{n_{2j-1}/2}c_{2i-1}\frac{m_{j_{2i-1}} }{n_{j_{2i-1}}}\sum_{n\in 
A_i}d_{\gamma_n}
$$
and consider the functional associated to the special node $\zeta_{n_{2j-1}}$, i.e. of the form
$$
e^{*}_{\zeta_{n_{2j-1}}}=\sum_{i=1}^{n_{2j-1}}d^{*}_{\zeta_{i}}+\frac{1}{m_{2j-1}} \sum_{i=1}^{n_{2j-1}/2}(\egs[\psi_{2i-1}]P_{J_{2i-1}}+\lambda_{2i}\egs[\psi_{2i}]P_{J_{2i}}).
$$
Then
\begin{align*}
&\norm[Ty]\geq e^{*}_{\zeta_{n_{2j-1}}}(Ty)
\\
&=\left(
\sum_{i=1}^{n_{2j-1}}d^{*}_{\zeta_{i}}+
\frac{1}{m_{2j-1}}
\sum_{i=1}^{n_{2j-1}/2}(\egs[\psi_{2i-1}]P_{J_{2i-1}}+
\lambda_{2i}\egs[\psi_{2i}]P_{J_{2i}}
)
\right)
(\frac{1}{n_{2j-1}}\sum_{i=1}^{n_{2j-1}/2}Tx_{2i-1})=\dots
\end{align*}
Notice that $J_{2i}\cap \Gamma_{\rank (\phi_{2i-1})}=\emptyset$, whereas by the choice of $R_n$ and the node $\phi_{2i-1}$ we have $\rng (Tx_{2i-1})\subset\Gamma_{\rank (\phi_{2i-1})}$. Therefore
\begin{align*}
&\dots=\left(
\sum_{i=1}^{n_{2j-1}}d^{*}_{\zeta_{i}}
\right)
(
\frac{1}{n_{2j-1}}\sum_{i=1}^{n_{2j-1}/2}Tx_{2i-1})+\frac{1}{n_{2j-1}m_{2j-1}}
\sum_{i=1}^{n_{2j-1}/2}
\egs[\psi_{2i-1}]P_{J_{2i-1}}(Tx_{2i-1})
\end{align*}
where in the last line the first sum disappears by the choice of $(q_{2i-1})$, as $\rank(\bdp(e^{*}_{\zeta_{n_{2j-1}}}))\cap\rank(Tx_{2i-1})=\emptyset$ for any $i$. Therefore we have
\begin{equation}\label{ty}
 \norm[Ty]\geq \frac{\delta}{28m_{2j-1}}.
\end{equation}

\

On the other hand we estimate $\norm[y]$. We shall prove that  $\norm[y]\leq 420/m^2_{2j-1}$ yielding for sufficiently big $j$ a contradiction. By (P2) and Lemma \ref{basisest} we get that $(x_{i})$ is 7-RIS. By Lemma~\ref{p56} it is enough to estimate $|e^{*}_{\beta}(z)|$, where $e^{*}_{\beta}$ is associated to a $(\Gamma, j)$-special sequence $(\delta_{i},\bar{z}_{i})_{i=1}^{a}$, and $z=\sum_{i\in J}x_{2i-1}$ for some interval 
$J\subset\{1,\dots,n_{2j-1}\}$. 

Let $e^{*}_{\beta}$ have the following form
$$
e_{\beta}^{*}=\sum_{i=1}^{a} 
d_{\tilde{\zeta}_i}^{*}+\frac{1}{m_{2j-1}}\sum_{i=1}^{\lfloor 
a/2\rfloor}(\tilde{\epsilon}_{2i-1}e_{\tilde{\psi}_{2i-1}}^{*} 
P_{\tilde{J}_{2i-1}}+\tilde{\lambda}_{2i}e^{*}_{\tilde{\psi}_{2i}} 
P_{\tilde{J}_{2i}})+\left[\tilde{\epsilon}_ae^*_{\psi_a}P_{\tilde{J}_a}\right]
$$
with $a\leq n_{2j-1}$, where the last term appears if $a$ is odd. Let $i_0=\max\{i\leq a: (\zeta_{i},\bar{x}_{i})=(\delta_{i},\bar{z}_{i})\}$ if such $i$ exists. We estimate $|e^{*}_{\beta}(z)|$ assuming $i_0$ is well-defined. We estimate separately $\abs{\sum_{i=1}^ad^{*}_{\tilde{\zeta}_i}(z)}$, $\abs{\mt (e^{*}_{\tilde{\zeta}_{i_0}})(z)}$ and $\abs{(\mt(e^{*}_{\beta})-\mt 
(e^{*}_{\tilde{\zeta}_{i_0}}))(z)}$. 

First notice that taking into account coordinates of $z$ with respect to the basis $(d_\gamma)$ and that $c_{2i-1}\leq 1$, we have 
\begin{equation}\label{op1}
\abs{\sum_{i=1}^ad^{*}_{\tilde{\zeta}_i}(z)}\leq 
n_{2j-1}\frac{m_{j_1}}{n_{j_1}} .
\end{equation}
Now consider the tree-analysis of $e^{*}_{\tilde{\zeta}_{i_0}}$, recall that it is compatible with the tree-analysis of $e^{*}_{\zeta_{i_0}}$. Then by the definition of a special node we have 
$$
\mt(e^{*}_{\zeta_{i_0}})=
\begin{cases}
\frac{1}{m_{2j-1}}\sum_{i=1}^{i_0/2}(\tilde{\epsilon}_{2i-1}\egs[\tilde{\psi}_{2i-1}]P_{J_{2i-1}}+\tilde{\lambda}_{2i}\egs[\psi_{2i}]P_{J_{2i}}) \quad 
&\textrm{if } i_0 \text{ even}\\
\frac{1}{m_{2j-1}}\sum_{i=1}^{\lfloor i_0/2\rfloor}(\tilde{\epsilon}_{2i-1}\egs[\tilde{\psi}_{2i-1}]P_{J_{2i-1}}
+\tilde{\lambda}_{2i}\egs[\psi_{2i}]P_{J_{2i}})+\tilde{\epsilon}_{i_0}\egs
[
\tilde{\psi}_{i_0}]P_{J_{i_0}} \quad &\textrm{if } i_0 \text{ odd}
\end{cases}
$$
where for each $2i-1\leq i_0$ we have
$$
e^{*}_{\tilde{\psi}_{2i-1}}=\sum_{n\in A_i}d^{*}_{\tilde{\xi}_n}+\frac{1}{m_{j_{2i-1}}}\sum_{n\in A_i}\tilde{\epsilon}_ne^{*}_{\tilde{\eta}_n}P_{I_n}.
$$
 
Notice that as $M\cap I_n=\emptyset$ for any $n$ and by the choice of $e^{*}_{\psi_{2i}}$ and ranks of $\xi_n$, thus also ranks of $\tilde{\xi}_n$, we get, assuming that $i_{0}$ is even,
\begin{align}\label{op2}
|\mt(e^{*}_{\zeta_{i_0}})(z)|&= 
|
\frac{1}{m_{2j-1}}\sum_{i=1}^{i_0/2}(\tilde{\epsilon}_{2i-1}\egs[\tilde{\psi}_{2i-1}]P_{J_{2i-1}}+\tilde{\lambda}_{2i}\egs[\psi_{2i}]P_{J_{2i}})(z)|
\\
&=|(\frac{1}{m_{2j-1}}\sum_{i=1}^{i_{0}/2}
\tilde{\epsilon}_{2i-1}\sum_{n\in A_i}d^{*}_{\tilde{\xi}_n})(\sum_{2i-1\in J}c_{2i-1}\frac{m_{j_{2i-1}} }{n_{j_{2i-1}}}\sum_{n\in A_i}d_{\gamma_n})|=0.
\notag
\end{align}
The same holds if $i_{0}$ is odd.

Now consider $\mt(e^{*}_{\beta})-\mt(e^{*}_{\zeta_{i_0}})$ assuming that $i_0<a$. Notice that 
\begin{enumerate}
\item $\w(\psi_s)\neq \w(\tilde{\psi}_i)$ for each $s,i>i_0$ provided at least one of the indices $s,i$ is bigger than $i_0+1$, 
\item  $(\mt(e^{*}_{\beta})-\mt(e^{*}_{\zeta_{i_0}}))(x_{2k-1})=0$ for any $2k-1\leq 
i_0$.
\end{enumerate}
Using Corollary~\ref{diffwb} for the terms $\sum_{i=i_0+1}^a\abs{e^{*}_{\tilde{\psi}_i}P_{\tilde{J}_i}(x_{2k-1} )}$ and that $|e^{*}_{\tilde{\psi}_{i_0+1}}P_{\tilde{J}_{i_0+1}}(x_{i_0+1})|\leq 4$, it follows that 
\begin{align}\label{op3}
\abs{(\mt(e^{*}_{\beta})-\mt(e^{*}_{\zeta_{i_0}}))(z)}
 &\leq 
\frac{1}{m_{2j-1}}\sum_{i=i_0+1}^a\sum_{2k-1=i_0+1}^{n_{2j-1}}\abs{e^{*}_{ \tilde{\psi}_i}P_{\tilde{J}_i}(x_{2k-1})}
 \\
&\leq
\frac{4}{m_{2j-1} }+
\frac{1}{m_{2j-1}}
n_{2j-1}
\frac{14}{m_{j_{i_{0}+1}}}
\leq
\frac{5}{m_{2j-1}}.
\notag
\end{align}

Therefore by \eqref{op1}, \eqref{op2}, \eqref{op3} and the choice of $j_1$ we have $\abs{e^{*}_{\beta}(z)}\leq 6/m_{2j-1}$, thus we can apply Lemma \ref{p56} obtaining that $\norm[y]\leq 60\cdot 7/m^{2}_{2j-1}$. For sufficiently big $j$ we obtain contradiction with \eqref{ty} and boundedness of $T$. 
\end{proof}

\begin{proposition} \label{d-to-RIS}
 Let $T: \mathfrak{X}_{Kus}\to \mathfrak{X}_{Kus}$ be a bounded operator. If $Td_{\gamma_n} \to 0$, then $Ty_n \to 0$ for every RIS $(y_n)_n$.
\end{proposition}

\begin{proof}
Take $T: \mathfrak{X}_{Kus}\to \mathfrak{X}_{Kus}$ with $Td_{\gamma_n}\to 0$ and suppose there are a normalized $C$-RIS $(y_n)_n$ and $\delta>0$ such that $\|Ty_n\| > \delta$ for all $n\in\N$. Passing to a subsequence we may assume as in the proof of Prop. \ref{dist} that
\[
 \max\rank R_n + 2 < \min\rank R_{n+1}\,\,\, \text{ where } 
R_n=\rng(Ty_n)\cup\rng(y_n).
 \]
Pick $(\mu_n)\subset\{\pm 1\}$ and nodes $(\psi_n)$ with $\mu_n e^{*}_{\psi_n}(Ty_n) > \delta$.

\textbf{Case 1}. There exist a constant $c>0$, an infinite set $M\subset\N$ and nodes $(\varphi_n)_{n\in M}$ such that $|e^{*}_{\varphi_n}(y_n)| > c$ and $e^{*}_{\varphi_n},e^{*}_{\psi_n}$ have compatible tree-analyses.

Pick signs $(\nu_n)_{n\in M}$ with $\nu_ne^{*}_{\varphi_n}(y_n)=|e^{*}_{\varphi_n}(y_n)| > c$ for each $n$. We may pass to a subsequence $(\gamma_{k_n})_n$ of $(\gamma_n)_n$ so that $\|Td_{\gamma_{k_n}}\|\leq 2^{-n}$ for all $n$. For a fixed $j\in\N$, $n_{2j+1}>200C$, we pick, as in Lemma \ref{dep1}, a $j$-dependent sequence $(\zeta_{i},x_{i})_i$ where $\zeta_{i}=(q_{i}+1,\zeta_{i-1},m_{2j-1},J_{i},1,\eta_{i})$, $i=1,\dots, n_{2j-1}$, with $\zeta_0=0$, satisfies
$$
\mt(e^{*}_{\eta_{2i-1}})=\frac{1}{m_{j_{2i-1}}}\sum_{n \in A_{2i-1}} \nu_n e^{*}_{\varphi_n}P_{I_n},\quad x_{2i-1}=\frac{c_{2i-1}m_{j_{2i-1}}}{n_{j_{2i-1}}}\sum_{n\in A_{2i-1}}y_n, \  \norm[x_{2i-1}]=\sfrac{1}{2}, 
$$
where $I_n=[\min R_n,\max R_n]$, so that the functional associated to the special node $\zeta_{n_{2j+1}}$ with mt-part of the form
$$
\mt(e^{*}_{\zeta_{n_{2j-1}}})=\frac{1}{m_{2j+1}}\sum_{i=1}^{n_{2j-1}/2} 
\left(e^{*}_{\eta_{2i-1}}P_{J_{2i-1}}+\lambda_{2i} e^{*}_{\eta_{2i}}P_{J_{2i}}\right),
$$
satisfies $J_{2i-1}\supset\rng (Tx_{2i-1})$, $J_{2i}\cap \rng (Tx_{2k-1})=\emptyset$ and $\rank (\bdp(e^{*}_{\zeta_{n_{2j+1}}}))\cap \rank (Tx_{2i-1})=\emptyset$ for any $i,k$.

From  Remark~\ref{rem6.3} we get
$$
\sfrac{1}{20C}\leq c_{2i-1}\leq 2.
$$
Using gaps between sets $R_n$ we pick nodes $(\xi_{2i-1})_{2i-1\leq n_{2j+1}}$, with
 \[
 \mt(e^{*}_{\xi_{2i-1}})=\frac{1}{m_{j_{2i-1}}}\sum_{n \in A_{2i-1}} \mu_n e^{*}_{\psi_n}P_{I_n}.
 \]
It follows that $e^{*}_{\xi_{2i-1}}(Tx_{2i-1}) > \sfrac{\delta}{20C}$ for each $i$. 

Notice also that for $x_{2i}=\frac{m_{j_{2i}}}{n_{j_{2i}}}\sum_{n\in A_{2i}}d_{\gamma_n}$, $A_{2i}\subset \{k_n: n\in\N\}$, by the condition on $(Td_{\gamma_{k_n}})$ we have $\|Tx_{2i}\| < \frac{m_{j_{2i}}}{n_{j_{2i}}} < 2^{-i}$ for each $i$.

Let $x = \frac{m_{2j-1}}{n_{2j-1}}\sum_{i=1}^{n_{2j-1}/2}x_{2i-1}$ and $ d = \frac{m_{2j-1}}{n_{2j-1}}\sum_{i=1}^{n_{2j-1}/2}x_{2i}$. We have 
\begin{equation}\label{d-to-ris1}
 \norm[Td]\leq \frac{m_{2j-1}}{n_{2j-1}}
\end{equation}
and by Lemma \ref{dep2}
\begin{equation}\label{d-to-ris2}
 \norm[x-d]\leq \frac{250}{m_{2j-1}^{2}}.
\end{equation}
On the other hand by the choice of $(\varphi_n)$ and $(\psi_n)$ there is a well-defined special node $\beta$, associated to the same $j$-special sequence as $\zeta_{n_{2j+1}}$ with 
$$
 \mt(e^{*}_\beta)=\frac{1}{m_{2j-1}}\sum_{i=1}^{n_{2j-1}/2} \left(e^{*}_{\xi_{2i-1}}P_{J_{2i-1}}+\tilde{\lambda}_{2i} 
e^{*}_{\eta_{2i}}P_{J_{2i}}\right),
$$
so that $\rank (\bdp(e^{*}_{\beta}))\cap \rank (Tx_{2i-1})=\emptyset$ for any $i$. Thus 
\begin{align*}
\norm[Tx]&\geq e^{*}_\beta(Tx)\geq \frac{\delta}{40C}
\end{align*}
which contradicts \eqref{d-to-ris1} and \eqref{d-to-ris2} for sufficiently big $j$ as $T$ is bounded.

\textbf{Case 2}. Case 1 does not hold. Applying this assumption for $c=n_{2j-1}^{-1}m_k^{-1}$, $k\in\N$, we pick inductively an increasing sequence $(p_k)\subset\N$ such that for any node $\varphi$ and $n>p_k$ so that $e^{*}_{\varphi},e^{*}_{\psi_n}$ have compatible 
tree-analyses we have $|e^{*}_{\varphi}(y_n)|\leq n_{2j-1}^{-1}m_k^{-1}$. Let $M=(p_k)_k$.

Now we repeat the proof of Prop. \ref{dist}, using $(y_n)$ instead of $(d_{\gamma_n})$. For a fixed $j\in\N$ we pick a $j$-dependent sequence $(\zeta_{i},x_{i})$, $\zeta_i=(q_i+1,\zeta_{i-1},m_{2j-1},J_i,1,\eta_i)$, $i=1,\dots, n_{2j-1}$, with $\zeta_0=0$, such that for each $i$ we have
$$
\mt(e^{*}_{\eta_{2i-1}})=\frac{1}{m_{j_{2i-1}}}\sum_{n\in A_i}\mu_ne^{*}_{\psi_n}P_{I_n},\,\,\
x_{2i-1}=\frac{c_{2i-1}m_{j_{2i-1}}}{n_{j_{2i-1}}}\sum_{n\in A_i}y_{n},\,\, \norm[x_{2i-1}]=\sfrac{1}{2},
$$ 
with $A_i\subset M$, $\# A_i=n_{j_{2i-1}}$ $J_{2i-1}=\rng (e^{*}_{\eta_{2i-1}})$, $J_{2i}\cap \supp x_{2k-1}=\emptyset$ for any $i,k$, $I_n=[\min R_n, \max R_n]$ and $\rank(\xi_n)=\max\rng R_n+1$ for any $n$. As  in the previous case, $\sfrac{1}{20C}\leq c_{2i-1}\leq 2$. Pick $j_1$ with $m_{j_1}/m_{j_1+1}\leq 1/n_{2j-1}^2$ and let 
$$
y=\frac{1}{n_{2j-1}}\sum_{i=1}^{n_{2j-1}/2}x_{2i-1}
$$
As in the proof of Prop. \ref{dist} it follows that 
\begin{equation}\label{ris-ty}
\norm[Ty]\geq e^{*}_{\zeta_{n_{2j-1}}}(y)\geq 
\frac{1}{m_{2j-1}n_{2j-1}}\sum_{i=1}^{n_{2j-1}/2}\frac{\delta}{2}c_{2i-1}\geq 
\frac{\delta}{80Cm_{2j-1}}.
\end{equation}
We shall estimate now $\norm[y]$. As before we consider a special node $\beta$ which is compatible with a $(\Gamma, j)$-special sequence $(\delta_{i},\bar{z}_{i})_{i=1}^{a}$, $a\leq n_{2j-1}$, and estimate $|e^{*}_{\beta}(z)|$ where $z=\sum_{i\in J}x_{2i-1}$ for some interval $J\subset\{1,\dots,n_{2j-1}\}$. Writing
$$
e_{\beta}^{*}=\sum_{i=1}^{a} 
d_{\tilde{\zeta}_i}^{*}+\frac{1}{m_{2j-1}}\sum_{i=1}^{\lfloor 
a/2\rfloor}(\tilde{\epsilon}_{2i-1}e_{\tilde{\eta}_{2i-1}}^{*} 
P_{\tilde{J}_{2i-1}}+\tilde{\lambda}_{2i}e^{*}_{\tilde{\eta}_{2i}} 
P_{\tilde{J}_{2i}})
$$
with $a\leq n_{2j-1}$ we pick as before $i_0=\min\{i\leq a: (\zeta_{i},x_{i})\ne(\delta_{i},z_{i})\}$ (if such 
$i$ exists) and estimate separately $\abs{\sum_{i=1}^ad^{*}_{\tilde{\zeta}_i}(w)}$, $\abs{\mt (e^{*}_{\tilde{\zeta}_{i_0}})(w)}$ and $\abs{(\mt(e^{*}_{\beta})-\mt (e^{*}_{\tilde{\zeta}_{i_0}}))(w)}$.

Repeating the reasoning of the proof of Prop. \ref{dist}, as $(y_n)$ have norm bounded by $1$ and all $\|d^{*}_{\tilde{\zeta}_i}\| \leq 3$, we obtain
\begin{equation}\label{op-ris1}
\abs{\sum_{i=1}^ad^{*}_{\tilde{\zeta}_i}(z)}\leq 
3\cdot 2n_{2j-1}\frac{m_{j_1}}{n_{j_1}} \leq \frac{1}{m_{2j-1}}.
\end{equation}
Using Corollary~\ref{diffw} and the fact that $|e^{*}_{\gamma}P_{I}(x_{i_{0}+1})|\leq 4$ we obtain that 
\begin{equation}\label{op-ris3}
 \abs{(\mt(e^{*}_{\beta})-\mt(e^{*}_{\tilde{\zeta}_{i_0}}))(z)}\leq 
 \frac{4}{m_{2j-1}}+ 2\frac{1}{m_{2j-1}} n_{2j-1}\frac{64C}{m_{j_{i_{0}+1}}}
 \leq 
\frac{5}{m_{2j-1}}
\end{equation}
using that  $m_{j_{1}}^{-1}<n_{2j-1}^{2}$  and  $n_{2j+1}>200C$.

Now consider $e^{*}_{\tilde{\zeta}_{i_0}}$, recall this functional and $e^{*}_{\zeta_{i_0}}$ have compatible tree-analyses. Therefore
$$
\mt(e^{*}_{\tilde{\zeta}_{i_0}})=
\begin{cases}
\frac{1}{m_{2j-1}}\sum_{i=1}^{i_0/2}(\tilde{\epsilon}_{2i-1}\egs[\tilde{\eta}_{2i-1}]P_{J_{2i-1}}+\tilde{\lambda}_{2i}\egs[\eta_{2i}]P_{J_{2i}}) \quad 
&\textrm{if } i_0 \text{ even}\\
\frac{1}{m_{2j-1}}\sum_{i=1}^{\lfloor 
i_0/2\rfloor}(\tilde{\epsilon}_{2i-1}\egs[\tilde{\eta}_{2i-1}]P_{J_{2i-1}}+\tilde{\lambda}_{2i}\egs[\eta_{2i}]P_{J_{2i}})+\tilde{\epsilon}_{i_0}\egs
[
\tilde{\eta}_{i_0}]P_{J_{i_0}} \quad &\textrm{if } i_0 \text{ odd}
\end{cases}
$$
where for each for each $2i-1\leq i_0$ we have
$$
e^{*}_{\tilde{\eta}_{2i-1}}=\sum_{n\in A_i}d^{*}_{\tilde{\xi}_n}+\frac{1}{m_{j_{2i-1}}}\sum_{n\in A_i}\tilde{\epsilon}_ne^{*}_{\varphi_n}P_{I_n}.
$$
By choice of the objects above we have
\begin{align*}
|\mt(e^{*}_{\tilde{\zeta}_{i_0}})(z)|&\leq 
\frac{1}{m_{2j-1}}|(\sum_{i=1}^{n_{2j-1}/2}\sum_{n\in A_i}d^{*}_{\tilde{\xi}_n})(\sum_{2i-1\in J}\frac{c_{2i-1}m_{j_{2i-1}} }{n_{j_{2i-1}}}\sum_{n\in A_i}y_n)|\\
&+\frac{1}{m_{2j-1}}\sum_{2i-1\in J}\frac{c_{2i-1}m_{j_{2i-1}} }{n_{j_{2i-1}}}\sum_{n\in A_i}|e^{*}_{\varphi_n}(y_n)|.
\end{align*}
As for each $n$ the nodes $\psi_n, \varphi_n$ have compatible tree-analyses the last sum can be estimated by $2m_{2j-1}^{-1}$. The first sum equals 0 by the condition on ranks of $\xi_n$, thus also $\tilde{\xi}_n$. Therefore we have
\begin{equation}\label{op-ris2}
 |\mt(e^{*}_{\tilde{\zeta}_{i_0}})(z)|\leq \frac{2}{m_{2j-1}}.
\end{equation}
As before by \eqref{op-ris1}, \eqref{op-ris3}, \eqref{op-ris2} we have $\abs{e^{*}_{\beta}(z)}\leq 8/m_{2j-1}$, thus we can apply Lemma \ref{p56} obtaining that $\norm[y]\leq 80C/m^{2}_{2j-1}$. For sufficiently big $j$ we obtain contradiction with \eqref{ris-ty} and boundedness of $T$. 

\end{proof}

\begin{theorem}
 Let $T: \mathfrak{X}_{Kus}\to \mathfrak{X}_{Kus}$ be a bounded operator. Then there exist a compact operator $K:\mathfrak{X}_{Kus}\to \mathfrak{X}_{Kus} $ and a scalar $\lambda$ such that $T=\lambda Id+K$.
\end{theorem}
\begin{proof}
 By Prop. \ref{dist} any $(d_{\gamma_n})_{n\in N}$ has a further subsequence $(d_{\gamma_n})_{n\in M}$ such that $Td_{\gamma_n}-\lambda d_{\gamma_n}\to 0$ as $M\ni n\to \infty$, for some $\lambda$. By Prop. \ref{spheres} there is a 
universal $\lambda$ so that $Td_{\gamma_n}-\lambda d_{\gamma_n}\to 0$ as $n\to \infty$. Applying Prop. \ref{d-to-RIS} to the operator $T-\lambda Id$ we get that $Ty_n-\lambda y_n\to 0$ for any RIS $(y_n)$ and thus, by Prop. \ref{fundamental}, for any bounded block sequence $(y_n)$. It follows that the operator $T-\lambda Id$ is compact. 
\end{proof}
The above theorem implies immediately the following. 
\begin{corollary}
 The space $\mathfrak{X}_{Kus}$ is indecomposable, i.e. it is not a direct sum of two its infinitely dimensional closed subspaces. 
\end{corollary}

\section{Unconditional saturation of the space $\mathfrak{X}_{Kus}$}
This section is devoted to the proof of saturation of the space $\mathfrak{X}_{Kus}$ by unconditional basic sequences. We follow the idea of the proof of the corresponding fact from \cite{AM} with additional work in order to control the bd-parts of norming functionals. Below we present a construction of unconditional sequences in $\mathfrak{X}_{Kus}$.

Fix a block subspace $Y\subset\mathfrak{X}_{Kus}$ and pick sequences $j_k<j_{k,1}<j_{k,2}<\dots< j_{k,n_{j_k}}$, $k\in\N$, with $(j_k)$ increasing, and a block sequence $(x_k)_k\subset Y$, with $x_k=\frac{m_{j_k}}{n_{j_k}}\sum_{i=1}^{n_{j_k}} x_{k,i}$ where for some fixed $C>2$ and for each $k\in\N$ the sequence $(x_{k,i})_i\subset Y$ is a $C-$RIS with parameters $(j_{k,i})_i$ chosen according to Lemma \ref{ris-exists} to satisfy $|d^{*}_\gamma(x_{k,i})|<1/n^2_{j_k}$ for any $i\leq n_{j_k}$ and $\gamma\in\Gamma$. Therefore
\begin{equation}\label{coordinates}
 |d^{*}_\gamma(x_k)|<C/n_{j_k}^2 \ \ \ \text{for any }k\in\N, \gamma\in\Gamma.
\end{equation}
We fix the sequence $(x_k)$ and the node $\gamma$ with the tree-analysis $(I_t, \epsilon_t,\eta_t)_{t\in\T}$ for the sequel.

Recall that $S_t$ denotes the set of immediate successors of $t$ in the tree $\T$. We order the sets $S_t$ with the order on $(I_s)_{s\in S_t}$ and we write $s_-$ for the immediate predecessor of $s$.

\begin{definition}
A couple of nodes $(\eta_{s_-}, \eta_s)$ is called a dependent couple with respect to $\gamma$ if $s_-,s\in S_t$, $\w(\eta_t)=m_{2j+1}^{-1}$ for some $j\in\N$ and $s$ is at the even position in the mt-part of $e^*_{\eta_t}$.
\end{definition}

 Let $\E_\gamma=\{s\in\T: (\eta_{s_-},\eta_s) \text{ is a dependent couple with respect to }\gamma\}$.

\begin{definition}
 For $k\in\N$ a couple of nodes $(\eta_{s_-}, \eta_s)$ is called a dependent couple with respect to $\gamma$ and $x_k$ if 
$(\eta_{s_-}, \eta_s)$ is a dependent couple with respect to $\gamma$ and moreover
 \[
 \min\supp(x_{k+1})>\max\supp(e^{*}_{s}P_{I_{s}}) \geq \min\supp(x_k),
 \]
 \[
 \max\supp(x_{k-1})\geq\min\supp(e^{*}_{s_-}P_{I_{s_-}}).
 \]
\end{definition}
\begin{remark}
Note that if $(\eta_{s_{-}},\eta_s)$, $(\eta_{t_{-}},\eta_t)$ are dependent couples then $t,s$ are incomparable.
\end{remark}

Let $\F_{\gamma}=\{s\in\T\mid (\eta_{s_-},\eta_s)\text{ is a dependent couple with respect to $\gamma$ and $x_k$ for some $k$}\}$ and let $Q_\gamma=\sum_{s\in\F_\gamma}P_{I_s}$. Then we define $y_k = Q_\gamma x_k$ and $x^\prime_k=x_k-y_k$. As our basis $(d_\gamma)_{\gamma\in\Gamma}$ is not unconditional, the projections $(Q_\gamma)_\gamma$ are not uniformly bounded. However, we have the following lemma that is proved along the lines of \cite{AM}. 
\begin{lemma} \label{action_on_ys}
 \begin{enumerate}
 \item[(i)] For every $k\in\N$ and $t\in\T$ we have $|e^{*}_{\eta_t} P_{I_t}(y_k)|\leq 10C/m_{j_k}$,
 \item[(ii)] For every $k\in\N$ and $t\in\T$ with $\w(\eta_t)<m_{j_k}^{-1}$ we have $|e^{*}_{\eta_t}P_{I_t}(x^\prime_k)|\leq 11C/m_{j_k}$.
 \end{enumerate}
\end{lemma}
\begin{proof}
Concerning $(i)$, notice first that for any $s\in\mathcal{F}_\gamma$ we have $|e^{*}_{\eta_s}P_{I_s}(x_k)|\leq 10C/m_{j_k}$. Indeed, for $\w(\eta_s)=m_{2j}$ for some $j$, we consider the following two cases. If $m_{2j}^{-1}<m_{j_k}^{-1}$ then the estimate follows by \eqref{bw}. If $m_{2j}^{-1}\geq m_{j_k}^{-1}$, then by the form of $e^{*}_{\eta_s}$ and \eqref{coordinates} we have 
$$
|e^{*}_{\eta_s}P_{I_s}(x_k)|\leq 2n_{2j}\max_{\gamma\in\Gamma}|d^{*}_{\gamma}(x_k)|\leq 2C/n_{j_k}
$$

Now, as each of the sets $\{s\in\F_\gamma\mid |s|=i, \rng(x_k)\cap I_s\neq\emptyset\}$, $i\in\N$, has at most two elements, we have
 \begin{eqnarray*}
 |e^{*}_{\eta_t} P_{I_t}(y_k)| & \leq & \sum_{s\in\F_\gamma}\left(\Pi_{t \preceq u \prec s}\w(\eta_u)\right)|e^{*}_{\eta_s}P_{I_s}(x_k)|\\
 & = & \sum_i\sum_{s\in\F_\gamma,|s|=i}\left(\Pi_{t \preceq u \prec s}\w(\eta_u)\right)|e^{*}_{\eta_s}P_{I_s}(x_k)|\\
 & \leq & \frac{20C}{m_{j_k}}\sum_i\frac{1}{m^i_1} = \frac{10C}{m_{j_k}}.
 \end{eqnarray*}
Condition $(ii)$ follows from Lemma~\ref{p56a} and $(i)$.
\end{proof}

\begin{lemma} \label{killing_signs}
 For  every choice of signs $(\delta_k)$ there exists a node $\widetilde \gamma\in\Gamma$ such that $Q_\gamma=Q_{\tilde{\gamma}}$ and $\epsilon\in\{\pm 1\}$ so that 
 \[
 |e^{*}_\gamma(x^\prime_k)-\epsilon e^{*}_{\widetilde\gamma}(\delta_kx^\prime_k)|\leq \frac{6C}{m_{j_k}} \text{ for any }k\in\N.
 \]
\end{lemma}
\begin{proof}
 Define 
 \begin{align*}
 D= & \{t\in\T\mid\rng(x_k) \cap\rng( e^{*}_tP_{I_t})\neq\emptyset\text{ for at most one }k \\
 & \text{ and if $t\in S_u$ then $\rng(x_i)\cap\rng( e^{*}_uP_{I_u})\ne\emptyset$ for at least two }i\}.
 \end{align*}
 
Since for every branch $b$ of $\T$ the set $b\cap D$ has exactly one element we can define a subtree $\T'$ of $\T$ such that $D$ is the set of terminal nodes for $\T'$. Notice that  $(\T\setminus\T')\cap\F_\gamma=\emptyset$.

If $\gamma\in D$, then we pick the unique $k_0$ with $\rng (e^{*}_\gamma)\cap \rng (x_{k_0})\neq \emptyset$ (as $I_\emptyset=[1,\max\Delta_{\rank(\gamma)}]$) and let $\tilde{\gamma}=\gamma$ and $\epsilon=\delta_{k_0}$. Then we have the estimate in the lemma for any $k\in\N$. 

Assume that $\gamma\not\in D$. Using backward induction on $\T'$ we shall define a node $\widetilde\gamma$ with a tree-analysis $(I_t,\tilde{\epsilon}_t,\widetilde\eta_t)_{t\in\T}$ and associated scalars $(\tilde{\lambda}_t)_{t\in\T}$, by modifying the nodes $(I_t,\epsilon_t,\eta_t)_{t\in\T'}$  and scalars $(\lambda_t)_{t\in\T'}$ starting from elements of $D$ such that 
 \begin{enumerate}
 \item[(T1)] $e^{*}_{\eta_t}$, $e^{*}_{\tilde{\eta}_t}$ have compatible tree-analyses for any $t\in\T'$, 
 \item[(T2)] $\F_{\widetilde\eta_t}=\F_{\eta_t}$ for any $t\in\T'$,
 \item[(T3)] $\tilde{\epsilon}_te^{*}_{\tilde{\eta}_t}P_{I_t}(\e_kx^\prime_k)=\epsilon_te^{*}_{\eta_t}P_{I_t}(x^\prime_k)$ for any $t\in D\setminus \E_\gamma$ and $k$, \\
 $\tilde{\lambda}_te^{*}_{\tilde{\eta}_t}P_{I_t}(\e_kx^\prime_k)=\lambda_te^{*}_{\eta_t}P_{I_t}(x^\prime_k)$ for any $t\in D\cap \E_\gamma$ and $k$,
 \item[(T4)] $\tilde{\epsilon}_t=\epsilon_t$ for any $t\in\T'\setminus D$.
 \end{enumerate}
Roughly speaking we need to modify only $\epsilon_{t}$, $t\in D$, changing signs of some of them. These modifications determine changes in the rest of the tree, i.e. $\eta_u$, $u\in\T'\setminus D$ according to the rules of producing nodes and Remark \ref{mt-part}.

\textbf{Step 1}. Take $t\in D$. 

\textbf{Case 1a}. $t\not \in \E_\gamma\cup\bigcup_{u\in \E_\gamma}S_u$. We set $\tilde{\eta}_t=\eta_t$ and $\tilde{\epsilon_t}=\delta_k\epsilon_t$, if $\rng(e^{*}_tP_{I_t})$ intersects $\rng(x_k)$ for some (unique) $k$, otherwise $\tilde{\epsilon}_t=\e_{m}\epsilon_{t}$ where $m=\min\{i:\rng e^{*}_{\eta_{t}}P_{I_{t}}\leq \rng(x_{i})\}$.

The condition (T3) follows straitforward.

\textbf{Case 1b}. 
$t\in \E_\gamma\cup\bigcup_{u\in \E_\gamma}S_u$. In this case we set $\tilde{\eta}_t=\eta_t$ and $\tilde{\epsilon_t}=\epsilon_t(=1)$. 
Moreover, for $t \in \E_\gamma$ we set $\tilde{\lambda}_t=\delta_{k}\lambda_t$. Such choice is possible since $Net_{1,q}$ is symmetric. It follows that
$$
\vert \tilde{\lambda}_{t}-\tilde{\epsilon}_{t^{-}}e^{*}_{\eta_{t^{-}}}(y_{2i-1}^{t})\vert=\vert\delta_{k}\lambda_{t}-\delta_{k}\epsilon_{t^{-}}e^{*}_{\eta_{t^{-}}}(y^{t}_{2i-1})\vert=\vert \lambda_{t}-\epsilon_{t^{-}}e^{*}_{\eta_{t^{-}}}(y^{t}_{2i-1})\vert
$$
where $(y_{i}^{t})_{i}$ are the vectors of the suitable special sequences.

In order to verify condition (T3) we consider two subcases.
\begin{enumerate}
 \item if $t\in \F_\gamma$ or $t\in S_u$ for some $u\in\E_\gamma$ (then $u\in\F_\gamma$), it follows that $\rng (e^{*}_{\eta_t}P_{I_t})\cap \rng x^\prime_k=\emptyset$ for any $k$ by the definition of $x^\prime_k$, thus we obtain (T3).
\item if $t\in \E_\gamma\setminus\F_\gamma$ and $\rng (e^{*}_{\eta_t}P_{I_t})\cap \rng x^\prime_k\neq\emptyset$ for some $k$, it follows that $e^{*}_{t_-}\in D$ as well and moreover $\rng (e^{*}_{\eta_{t_-}}P_{I_{t_-}})$ either intersects only $\rng x_k$ or intersects no $\rng x_i$. In both cases $\tilde{\epsilon}_{t_-}=\e_k\epsilon_{t_-}$ and so $\tilde{\lambda}_t=\e_k\lambda_t$ and (T3) holds.

\end{enumerate}

Notice that in either case conditions (T1)-(T2) and (T4) are straitforwardly satisfied.

\textbf{Step 2}. Now we define inductively nodes in $t\in \T'\setminus D$. Take $t\in \T'\setminus D$ and assume we have defined $(\tilde{\epsilon}_s,\tilde{\eta}_s, I_s)_{s\in S_t}$ satisfyng (T1)-(T4). In all cases we let $\tilde{\epsilon_t}=\epsilon_t$, thus (T4) is satisfied. Notice that $t\not\in\bigcup_{u\in \E_\gamma}S_u$.

\textbf{Case 2a}. $t\in \E_\gamma$. In this case we set $\tilde{\eta}_t=\eta_t$. Obviously we have (T1)-(T2).

\textbf{Case 2b}. $t\not \in \E_\gamma$, $\w(\eta_t)=m_{2j}^{-1}$. Then using Remark \ref{mt-part} (1) we define $\tilde{\eta}_t$ so that 
$$
\mt(e^{*}_{\tilde{\eta}_{t}})=\frac{1}{m_{2j}}\sum_{s\in S_{t}}\tilde{\epsilon}_se^{*}_{\tilde{\eta}_{s}}P_{I_{s}}.
$$
By definition we have (T1)-(T2).

\textbf{Case 2c}. $\w(\eta_t)=m_{2j+1}^{-1}$, with $\eta_t$ compatible with a $(\Gamma,j)$-special sequence $(\bar{x}_t,\bar{\eta}_t)$.
Then using Remark \ref{mt-part} (2) we define a special node $\tilde{\eta}_t$ which is compatible with the same 
$(\Gamma,j)$-special sequence $(\bar{x}_t,\bar{\eta}_t)$ so that 
$$
\mt(e^{*}_{\tilde{\eta}_{t}})=\frac{1}{m_{2j+1}}\sum_{s\in S_{t}\cap \E_\gamma}(\tilde{\epsilon}_{s_-}e^{*}_{\tilde{\eta}_{s_-}}P_{I_{s_-}}+\tilde{\lambda}_se^{*}_{\tilde{\eta}_{s}}P_{I_s}).
$$

By definition we have (T1)-(T2).

Let $\tilde{\gamma}=\tilde{\eta}_\emptyset$. Notice that by conditions (T1)-(T2) we have $Q_{\tilde{\gamma}}=Q_\gamma$.

 \ 
 
Now we proceed to show the estimate part of the lemma. Fix $k\in\N$. For any non-terminal $u\in\T$ let 
$$
S_{u,k}:=\{s\in S_{u}\mid \rng(x_k)\cap \rng(e^{*}_sP_{I_s})\neq\emptyset\}.
$$
Let $G$ be the set of minimal nodes $u$ of $\T^\prime$ with $u\in D$ or $\w(\eta_{u})<m_{j_k}^{-1}$. By $T''$ denote the subtree of $T^\prime$ with the terminal nodes in $G$.

We shall prove by induction starting from $G$ that for any $u\in\T''$ we have 
\begin{align}\label{est}
|\epsilon_ue^{*}_{\eta_{u}}P_{I_{u}}(x^\prime_k)-\tilde{\epsilon}_ue^{*}_{\widetilde\eta_{u}}P_{I_{u}}(\delta_kx^\prime_k)|\leq \frac{22C}{m_{j_k}}.
\end{align}
This will end the proof as it follows by (T4) that $|\epsilon_\emptyset e^{*}_{\eta_\emptyset}(x^\prime_k)-\tilde{\epsilon}_\emptyset e^{*}_{\tilde{\eta}_\emptyset}(\delta_kx^\prime_k)|=|e^{*}_\gamma(x^\prime_k)-e^{*}_{\tilde{\gamma}}(\delta_kx^\prime_k)|$. Thus taking $\epsilon=1$ we obtain the estimate of the lemma.

\textbf{Step 1}. $u\in G$. If $\w(\eta_{u})<m_{j_k}^{-1}$ then the estimate \eqref{est} holds true by Lemma \ref{action_on_ys} (ii). If $u\in D$ then the estimate \eqref{est} holds true by (T3).

\textbf{Step 2}. $u\in\T''\setminus G$. In particular $\w(\eta_{u})\geq m_{j_k}^{-1}$. Obviously $S_u\subset T''$.

\textbf{Case 2a}. $\w(\eta_{u})=m_{2j}^{-1}$. We estimate, using (T3) for $s\in S_{u,k}\cap D$
\begin{align*}
|e^{*}_{\widetilde\eta_{u}}P_{I_u}(\e_{k}x_{k}^{\prime})-e^{*}_{\eta_{u}}P_{I_{u}} (x^\prime_k)|
&=|\left(\sum_{s\in S_u}d_{\tilde{\xi}_{s}}^{*}+\frac{1}{m_{2j}}\sum_{s\in S_{u,k}\cap D }\tilde{\epsilon}_se^{*}_{ \tilde{\eta}_{s} } P_{I_s}+\frac{1}{m_{2j}}\sum_{s\in S_{u,k}\setminus D }\tilde{\epsilon}_se^{*}_{ \tilde{\eta}_{s} } 
P_{I_s}\right)(\e_{k}x_{k}^{\prime})\\
&-\left(\sum_{s\in S_{u}}d_{\xi_{s}}^{*}+\frac{1}{m_{2j}}\sum_{s\in S_{u,k}\cap D }\epsilon_se^{*}_{\eta_{s}}P_{I_s}+\frac{1}{m_{2j}}\sum_{s\in S_{u,k}\setminus D }\epsilon_se^{*}_{ \eta_{s}}P_{I_s} 
\right)(x_{k}^{\prime})|
 \\
&\leq |\sum_{s\in S_{u} } d_{\tilde{\xi}_{s}}^{*}(\e_{k}x_{k}^{\prime})|+|\sum_{s\in S_{u} } d_{\xi_{s}}^{*}(x_{k}^{\prime})|+\frac{1}{m_{2j}}\sum_{s\in S_{u,k}\setminus D} |\tilde{\epsilon}_se^{*}_{\tilde{\eta}_s}P_{I_s}(\e_kx_k^\prime) -\epsilon_se^{*}_{\eta_s}P_{I_s}(x_k^\prime)|\\
&\leq \dots
\end{align*}
The first two sums are  estimated using \eqref{coordinates} and $\# S_{u}\leq n_{2j}\leq n_{j_k}$, for the third element use the inductive hypothesis and the fact that $\# (S_{u,k}\setminus D)\leq 2$, obtaining the following
\begin{align*}
\dots&\leq
2n_{2j}\frac{C}{n_{j_{k}}^2}+\frac{2}{m_{2j}}\cdot\frac{22C}{m_{j_k}}\leq\frac{22C}{m_{j_k}}.
\end{align*}

\textbf{Case 2b}. $\w(\eta_{u})=m_{2j+1}^{-1}$. Recall that by (T3) we have $\epsilon_{s_-}e^{*}_{\eta_{s_-}}P_{I_{s_-}}(x^\prime_k)=\tilde{\epsilon}_{s_-}
e^{*}_{\tilde{\eta}_{s_-}}I_{s_-}(\delta_kx^\prime_k) $ for any $s\in S_u\cap\E_\gamma$ with $s_-\in D$ and 
$\lambda_se^{*}_{\eta_s}P_{I_s}(x^\prime_k)=\tilde{\lambda}_se^{*}_{\tilde{\eta}_s} I_s(\delta_kx^\prime_k) $ for any $s\in S_u\cap\E_\gamma\cap D$. Moreover $\E_\gamma\setminus D\subset\F_\gamma$ thus $e^{*}_{\eta_s}P_{I_s}(x^\prime_k)=0=e^{*}_{\tilde{\eta}_s}P_{I_s}
(\delta_kx^\prime_k)$ for any $s\in (S_u\cap\E_\gamma)\setminus D$. Therefore we have
\begin{align*}
|e^{*}_{\wth_{u}}P_{I_u}(\e_{k}x_{k}^{\prime})&-e^{*}_{\eta_{u}}P_{I_{u}}(x^\prime_k)|\\
&=|\left(\sum_{s\in S_{u}}d_{\tilde{\xi}_{s}}^{*}+\frac{1}{m_{2j+1}} \sum_{
s_-\in S_{u,k}, s\in\E_\gamma}\tilde{\epsilon}_{s_-}e^{*}_{ \tilde{\eta}_{s_{-} }}P_{I_s}+
\frac{1}{m_{2j+1}} \sum_{ s\in S_{u,k}\cap\E_\gamma}\tilde{\lambda}_s e^{*}_{\wth_{s} }P_{I_s}\right)(\e_{k}x_{k}^{\prime})
 \\
&-\left(\sum_{s\in S_{u}}d_{\xi_{s}}^{*}+\frac{1}{m_{2j+1}} \sum_{ s_-\in S_{u,k},s\in\E_\gamma} \epsilon_{s_-}e^{*}_{ \eta_{s_{-} }}P_{I_s}+\frac{1}{m_{2j+1}} \sum_{ s\in S_{u,k}\cap\E_\gamma} \lambda_s e^{*}_{\eta_{s} }P_{I_s}\right)(x_{k}^{\prime})|
\\
&=|\left(\sum_{s\in S_{u}}d_{\tilde{\xi}_{s}}^{*}+\frac{1}{m_{2j+1}} \sum_{ s_-\in S_{u,k}\setminus D, s\in\E_\gamma}\tilde{\epsilon}_{s_-}e^{*}_{ \tilde{\eta}_{s_{-} }}P_{I_s}\right)(\e_{k}x_{k}^{\prime})
 \\
&-\left(\sum_{s\in S_{u}}d_{\xi_{s}}^{*}+\frac{1}{m_{2j+1}} \sum_{ s_-\in S_{u,k}\setminus D,s\in\E_\gamma} \epsilon_{s_-}e^{*}_{ \eta_{s_{-} }}P_{I_s}\right)(x_{k}^{\prime})|
\\
&\leq |\sum_{s\in S_{u} } d_{\tilde{\xi}_{s} }^{*}(\e_{k}x_{k}^{\prime})|+|\sum_{s\in S_{u} } d_{\xi_{s}
}^{*}(x_{k}^{\prime})|+
\\
&\hspace{3cm}
+\frac{1}{m_{2j+1}}\sum_{s_-\in S_{u,k}\setminus D, s\in\E_\gamma}|\tilde{\epsilon_{s_-}}e^{*}_{\tilde{\eta}_{s_-}}P_{I_{s_-}}
(\e_kx_k^\prime) -\epsilon_{s_-}e^{*}_{\eta_{s_-}}P_{I_{s_-}}(x_k^\prime)|
\\
&\leq\dots
\end{align*}
Proceeding as in Case 2a we obtain
\begin{align*}
 \dots&\leq 2n_{2j+1}\frac{C}{n_{j_{k}}^2}+\frac{2}{m_{2j+1}}\cdot \frac{22C}{m_{j_k}}\leq\frac{22C}{m_{j_k}}.
\end{align*}

\end{proof}

\begin{theorem} \label{us}
The space $\mathfrak{X}_{Kus}$ is unconditionally saturated.
\end{theorem}
\begin{proof} 
In every block subspace of $\mathfrak{X}_{Kus}$ pick a sequence $(x_k)_k$ as above with $m_{j_1}>400C$. We claim that such a sequence is unconditional. To this end consider a finite sequence of scalars $(a_k)$ with $\|\sum_ka_kx_k\|=1$ and $(\delta_k)\subset \{\pm 1\}$. We want to estimate the norm of the vector $\sum_k\delta_ka_kx_k$. Take $\gamma\in\Gamma$ with $e^{*}_{\gamma}(\sum_ka_kx_k)\geq \sfrac{3}{4}$. Define $Q_\gamma$, $(y_k)$ and $(x^\prime_k)$ and consider $\tilde{\gamma}$ and $\epsilon$ provided by Lemma \ref{killing_signs}. Notice that as $Q_{\tilde{\gamma}}=Q_\gamma$, the projection $Q_{\tilde{\gamma}}$ defines also $(y_k)$ and $(x^\prime_k)$. Estimate, applying Lemma \ref{killing_signs} and Lemma \ref{action_on_ys} (1) both for $\gamma$ and $\tilde{\gamma}$, as follows
\begin{align*}
 |e^{*}_\gamma(\sum_ka_kx_k)&-\epsilon e^{*}_{\tilde{\gamma}}(\sum_k\delta_ka_kx_k)| 
 \\
 &\leq |e^{*}_\gamma(\sum_ka_kx^\prime_k)-\epsilon e^{*}_{\tilde{\gamma}}(\sum_k\delta_k a_kx^\prime_k)|+|e^{*}_\gamma(\sum_k a_ky_k)|+|e^{*}_{\tilde{\gamma}}(\sum_k\delta_ka_ky_k)|
 \\
 &\leq \sum_k|a_k||e^{*}_{\gamma}(x^\prime_k)-\epsilon e^{*}_{\tilde{\gamma}}(\delta_k x^\prime_k)|+\sum_k|a_k||e^{*}_\gamma(y_k)|+\sum_k|a_k||e^{*}_{\tilde{\gamma}}(\delta_ky_k)|
 \\
 &\leq 4\cdot 24C\sum_k m_{j_k}^{-1}\leq 200Cm_{j_1}^{-1}\leq \sfrac{1}{2}
 \end{align*}
 where in the last line we use the fact that each $|a_k|$ is dominated by twice the basic constant of the basis $(d_\gamma)$. Therefore $\|\sum_k\delta_ka_kx_k\|\geq |e^{*}_{\tilde{\gamma}}(\sum_k\delta_ka_kx_k)|\geq \sfrac{1}{4}$, which ends the proof. 
\end{proof}

\end{document}